\documentclass[11pt]{amsart}
 \usepackage{amsmath,amssymb,enumerate}
 \usepackage[all]{xy}
 \newtheorem{prop}{Proposition}[section]
 \newtheorem{coro}[prop]{Corollary}

 \newtheorem{lem}[prop]{Lemma}
 \newtheorem{rem}[prop]{Remark}
 \newtheorem{exa}[prop]{Example}
 
  \newtheorem{defi}[prop]{Definition}
 \newtheorem{theo}[prop]{Theorem}

 \newcommand{\B}{\mathbb B}
 
 \newcommand{\R}{\mathbb R}
 \newcommand{\Q}{\mathbb Q}
 \newcommand{\C}{\mathbb C}
 \newcommand{\N}{\mathbb N}
 \newcommand{\Z}{\mathbb Z}

 \newcommand{\e}{\varepsilon}
 \newcommand{\f}{\varphi}
 
 \newcommand{\p}{\psi}

 \numberwithin{equation}{section}

\begin{document}

  \title[Parabolic complex Monge-Amp\`ere equations III]{Convergence of weak K\"ahler-Ricci Flows on minimal models of positive Kodaira dimension}

\setcounter{tocdepth}{1}

  \author{ Philippe Eyssidieux, Vincent Guedj, Ahmed Zeriahi} 

\address{Universit\'e Joseph Fourier et Institut Universitaire de France, France}

\email{Philippe.Eyssidieux@ujf-grenoble.fr }

\address{Institut de Math\'ematiques de Toulouse et Institut Universitaire de France,   \\ Universit\'e Paul Sabatier \\
118 route de Narbonne \\
F-31062 Toulouse cedex 09\\}

\email{vincent.guedj@math.univ-toulouse.fr}

\address{Institut de Math\'ematiques de Toulouse,   \\ Universit\'e Paul Sabatier \\
118 route de Narbonne \\
F-31062 Toulouse cedex 09\\}

\email{zeriahi@math.univ-toulouse.fr}

 \date{\today}

 \begin{abstract}  
Studying the behavior of the K\"ahler-Ricci flow on mildly singular varieties, one is naturally lead to study weak solutions of degenerate parabolic complex Monge-Amp\`ere equations.

In this article, the third of a series on this subject, we study the long term 
behavior of the normalized K\"ahler-Ricci flow on mildly singular varieties of positive Kodaira dimension, 
generalizing results of Song and Tian who dealt with smooth minimal models.
\end{abstract}

 \maketitle

\tableofcontents


\section{Introduction}

The (normalized) K\"ahler-Ricci flow on a compact K\"ahler manifold $X$ starting at a K\"ahler form $\omega_0$ is 
the initial value problem for the  evolution of   a 
smooth family of K\"ahler forms $ (\omega_t)_{t\ge 0}$ given by 
$$\frac{\partial \omega}{\partial t} = -Ric(\omega) -\omega. 
$$
This flow has infinite existence time if and only if $K_X$ is nef \cite{TZ06}. 

Assuming the abundance conjecture for $X$, the asymptotic behaviour as $t\to \infty$ has been understood
in \cite{ST07, ST12}. 
The flow converges weakly towards a closed positive $(1,1)$-current $T_{can} \in \{ K_X \}$ independent of the initial metric $\omega_0$. 
It appears as  the curvature of a singular hermitian metric giving rise to a canonical 
absolutely continuous measure (or volume form) 
$v_{can}$. If $j: X \to J$ is the natural model of the Iitaka fibration of $X$, that is 
the proper holomorphic map appearing as the Stein factorization of the regular map attached to the linear system $|NK_X|$ for $N\in \N^*$ divisible enough, one has 
also 
$T_{can}=j^*\theta_{can}$ where $\theta_{can}$ is a canonical twisted K\"ahler-Einstein metric, namely a closed positive current on $J$ satisfying a 
complex Monge-Amp\`ere equation in the sense of \cite{EGZ09}. 

The basic ideas of the Minimal Model Program indicate that these results should generalize to singular K\"ahler varieties with terminal singularities 
(a rather mild sort of singularities) and also to klt pairs. 
The study of  the K\"ahler-Ricci flow on these mildly singular varieties was undertaken by 
 Song and Tian in \cite{ST09}.  It requires 
a theory of weak solutions for certain  degenerate parabolic complex Monge-Amp\`ere equations modelled on 
\begin{equation}\label{krf}
 \frac{\partial \phi}{\partial t}+\phi= \log \frac{(dd^c\phi)^n}{v},
\end{equation}
where $v$ is a volume form and $\phi$ a time-dependent K\"ahler potential.

The approach in \cite{ST09} is to regularize the equation
and take limits of the solutions of the regularized equation with uniform
higher order estimates. We believe that   
a zeroth order approach should be both simpler and more efficient, namely an approach  
using the PDE method of viscosity solutions for degenerate parabolic equations. 
Using this method, basic existence theorems for the K\"ahler Ricci flow on klt pairs 
have been established in \cite{EGZ16}.
 
 The present article, the third one of a series on this subject, 
 aims at studying the long term  
behavior of the normalized K\"ahler-Ricci flow on terminal abundant K\"ahler varieties of positive Kodaira dimension.

We begin by developping a generalization of Song-Tian's canonical volume form and twisted K\"ahler-Einstein metric to singular varieties and pairs:

\medskip
\noindent {\bf Theorem A.}
{\it 
Let $X$ be a compact $n$-dimensional K\"ahler variety  and $\Delta$ be an effective $\Q$-divisor such that 
the pair $(X,\Delta)$ has klt singularities and semi-ample canonical bundle.
Then the Iitaka variety $J$ of $(X,\Delta)$ carries a canonical twisted K\"ahler-Einstein metric $\theta_{can}$ with continuous potentials. If $X$ is projective, $\theta_{can}$ is
smooth on a non empty Zariski open subset of $J$.
}
\medskip

 See section 2 for a more precise statement and further generalizations.  
 The $C^0$-existence theorem is relatively simple in our approach and follows from \cite{EGZ09,CGZ13}. 
 The  generic smoothness of $\theta_{can}$ is  more technical 
 and established via P\u{a}un's Laplacian estimate \cite{Paun08},
 unlike Song and Tian who established the case when $X$ is smooth and $\Delta=\emptyset$ by 
 a direct perturbation method. In our approach,  P\u{a}un's conditions are established using Kawamata's canonical bundle formula. 
The projectivity assumption can probably be removed
 and comes from our reliance on Ambro's study of lc-trivial fibrations \cite{A}.

 In the context of Theorem A, we denote by $T_{can}$ be the pull back of $\theta_{can}$ by the Iitaka map. 
 In the general type case the canonical current $T_{can}$ coincides with the singular K\"ahler-Einstein metric of \cite{EGZ09}. In the Calabi-Yau case it reduces to the zero current.
 
 \smallskip

Generalizing results of Song and Tian who dealt with smooth minimal models \cite{ST12}, we conjecture :

 \medskip
\noindent {\bf Conjecture B.}
{\it 
Let $(X,\Delta)$ be a compact $n$-dimensional K\"ahler klt pair  
  with canonical singularities and semi-ample canonical bundle.
Then the normalized K\"ahler-Ricci flow
continuously deforms any 
initial K\"ahler form  $\omega_0$ towards the canonical current
$T_{can}$.
}
\medskip

 Since we do not have a complete proof even in the case where $\Delta=\emptyset$ and 
 $X$ has only canonical singularities, we will limit the discussion to this case.

When the Kodaira dimension $\kappa$  is either zero or maximal 
Conjecture B has been established in \cite{EGZ15}, following earlier works of Cao \cite{Cao85}, Tsuji \cite{Ts88},
Tian-Zhang \cite{TZ06} and Song-Yuan \cite{SY12}.
In this case the canonical current is the singular K\"ahler-Einstein metric constructed in \cite{EGZ09}.

We focus in this article on the harder case of intermediate Kodaira dimension $0<\kappa<n$.
The Minimal Model Program predicts  that any compact K\"ahler manifold 
of positive Kodaira dimension $\kappa$
is birationally equivalent to a minimal model $X$ 
with terminal (hence canonical) singularities.
The abundance conjecture then stipulates that the canonical bundle of $X$ is semi-ample,
in accordance with our hypotheses.
This program is known to hold in dimension $\leq 3$ \cite{CPH}.

Denote by $j:X\to J$  the Iitaka fibration of $X$. The semi-flat current $\omega_{SF}=\omega_0+dd^c \rho$ is constructed as follows : 
a generic fiber $X_y$ of $j$ is a $\Q$-Calabi-Yau variety, 
then ${\omega_{SF}}_{|X_y}={\omega_0}_{|X_y}+dd^c \rho_y$ denote the unique Ricci-flat current in $\{\omega_0\}_{|X_y}$,
where $\rho_y$ is normalized so that $\int_{X_y} \rho_y \, {\omega_0}_{|X_y}^{n-\kappa}=0$ \cite{EGZ09}.
We reduce Conjecture B to   a regularity statement for the {\em semi-flat current} $\omega_{SF}$ (aka the fiberwise Ricci flat metric) on $X$ :

 \medskip
\noindent {\bf Conjecture C.}
{\it The semi-flat current
$\omega_{SF}$ is smooth on a non empty Zariski open subset of $X$.
}
\medskip

When the generic fibers $X_y$ are smooth,  it follows from Yau's theorem \cite{Yau78}
and the implicit function theorem
that $\rho$ is indeed smooth on 
the union of regular fibers.
This is is also automatic if $\kappa(X)=\dim(X)-1$. 
If $\kappa(X)=\dim(X)-2$, the generic fiber has only isolated quotient singularities and 
orbifold techniques readily yield conjecture C in this case. 

Conjecture C can be seen as a relative version of \cite{EGZ09}. When $X$ is smooth,
the semi flat-current has been recently studied by Choi \cite{Choi15}
who proved that $\omega_{SF} \ge 0$ on the union of regular fibers of $j$. We expect both statements  to hold
for the fiberwise K\"ahler-Einstein metric of a fibration whose general fiber is canonically polarized.
One may also speculate that $\omega_{SF}$  extends as a closed positive current on the whole of $X$. 

Our second main result shows that Conjecture C implies Conjecture B in the following context :

\medskip
\noindent {\bf Theorem D.}
{\it 
Let $X$ be a compact $n$-dimensional K\"ahler variety  
  with  canonical singularities and semi-ample canonical bundle.
  Assume the semi-flat current $\omega_{SF}$ is smooth on a non empty Zariski open subset of $X$.

Then the normalized K\"ahler-Ricci flow
continuously deforms any 
initial K\"ahler form  $\omega_0$ (resp.
any K\"ahler current with continuous potentials in $\{\omega_0\}$) towards the canonical current
$T_{can}$.
}
\medskip

Note that our assumptions are satisfied hence our result applies to any minimal $3$-fold.
The convergence holds at the level of potentials; it is locally uniform on a Zariski open subset of $X$
(away from the singular locus of $X$ and the singular fibers of the Iitaka fibration). 
More precise information on the convergence have been provided in recent works when $X$ is smooth,
and when the fibers are tori (see notably
\cite{Tos10,TWY14,TZ15}).

\medskip

We now describe the main tool for the proof
Theorem D, a viscosity comparison principle
for compact manifolds with boundary, which applies to quite general complex Monge-Amp\`ere flows.
We study here the  flows
\begin{equation} \label{eq:krfglobal}
 e^{\dot{\f}_t + F (t,x,\f) } \mu (x) - (\omega_t +dd^c \f_t)^n = 0, 
\end{equation}
in $M_T=(0,T) \times M$, where $T \in ]0,+\infty]$ and
  \begin{itemize}
 \item $\overline{M}$ is a compact K\"ahler $n$-dimensional manifold with interior $M$ and boundary $\partial M = \overline M \setminus M$;
  \item $\omega(t,x)$ is a continuous family of semi-positive $(1,1)$-forms on $\overline{M}$, 
uniformly bounded from below by a semi-positive and big form $\theta(x)$;
\item we always assume that $M$ is contained in the ample locus of $\theta$;
  \item $F (t,x,r)$ is continuous in $[0,T[ \times \overline{M} \times \R$ and non decreasing in $r$,
$(t,x,r) \mapsto F(t,x,r)$ is uniformly Lipschitz  in $r$ and 
$(t, x) \mapsto  F (t, x, 0)$ is uniformly bounded from above.
  \item $\mu (x)\geq 0$ is a bounded continuous volume form on $\overline{M}$,
   \item $\f : [0,T[ \times \overline{M} \rightarrow \R$ is the unknown function, with $\f_t: = \f (t,\cdot)$. 
   \end{itemize}

\medskip
\noindent {\bf Theorem E.}
{\it 
 Assume  $t \mapsto \omega_t$ is regular.
Let $u(t,x)$ (resp. $v(t,x)$) be a subsolution (resp. a supersolution) to the equation (\ref{eq:krfglobal})
in $M_T$. If $u \leq v$ on $\partial_P M$, then $u \leq v$ on $M_T= (0,T) \times M$.
}
\medskip

Here $\partial_P M$ denotes the parabolic boundary of $M_T$,
\begin{equation} \label{eq:parabd}
\partial_P M=\{0\} \times \overline{M} \, \bigcup  \, [0,T) \times \partial M ,
\end{equation}
and we say $t\mapsto \omega_t$ is regular, if the following holds: 
for every $ \epsilon >0$ there exists $E(\e)> 0 $ such  that $\forall t\in [0,T-2\e],  \ \forall t'\in ]t-\e,t+\e[$,
\begin{equation} \label{eq:techcond}
  \ (1+E(\epsilon)) \omega_t \ge \omega_{t'} \ge (1-E(\epsilon) ) \omega_t 
\end{equation}
and $E(\e) \to 0$ as $\e \to 0$. 
This technical condition is satisfied in the following particular case of interest: if $\pi: X\to Y$ is a bimeromorphic morphism onto a normal K\"ahler variety and $\omega_t^Y$ is a continuous family 
of {\it {smooth}} K\"ahler forms on $Y$, then $\omega_t=\pi^*\omega_t^Y$  is regular
(see \cite[Lemma 2.7]{EGZ16}).

When $M$ is a bounded pseudonconvex domain of $\C^n$, Theorem E  was  obtained in 
\cite{EGZ15}. When $M$ is a compact K\"ahler manifold without boundary, it is the 
main technical result of \cite{EGZ16}. 
We have been informed by J.Streets that the proof of this result \cite[Theorem A]{EGZ16} is not correct as it stands.
We provide here a different approach which yields an alternative proof of the global comparison principle in the ample locus
and is valid in the present context of manifolds with boundary.
This extension of the comparison principle is both useful and necessary to construct sub/supersolutions to the K\"ahler-Ricci flow 
 (see Propositions \ref{pro:subsol2}, \ref{pro:supsolkappageneral}).

 \smallskip

We finally describe the precise contents of the article.
Section \ref{sec:st} studies Song-Tian's  canonical volume form and current in a general context,
proving {\it Theorem A}.
In section \ref{sec:visc} we recall the parabolic viscosity tools introduced in
\cite{EGZ15,EGZ16} and prove {\it Theorem E}.
We finally prove  {\it Theorem D} in section \ref{sec:longterm}
by establishing global uniform bounds and then 
constructing  appropriate viscosity sub/ supersolutions in  arbitrary large subsets
of a fixed Zariski open set.

 \medskip

\noindent {\bf Acknowledgements.} 
We thank S. Druel and E. Floris for useful conversations.
We also would like to thank Jeff Streets for pointing out a problem in the original proof of the comparison principle.

\medskip

\section{The Song-Tian canonical volume form}  \label{sec:st}

In this section, we  simplify the exposition of the  construction of  the Song-Tian canonical volume form and current \cite{ST12} as a direct application of \cite{EGZ09}, 
and we extend it first to the case of abundant klt pairs then to sub-klt lc-trivial fibrations. The proof of  generic smoothness for klt pairs does indeed require
the framework to be enlarged to lc-trivial fibrations.

\subsection{Push forward of volume forms by holomorphic mappings} Let us begin by an easy lemma. 

\begin{lem}\label{pushforward1}
 Let $(X,\omega_X)$ (resp. $(Y,\omega_Y)$) be normal  K\"ahler spaces. Let $F: X\to Y$ be a proper
surjective holomorphic mapping.
Then, for every  precompact open subset $U\subset Y$ there exists a constant
 $\epsilon=\epsilon_U >0$ such that $F_*  \omega_X^{\dim X} = \sigma \omega_Y^{\dim Y}$ 
 with $\sigma\in L^{1+\epsilon}(U, \omega_Y^{\dim Y})$
\end{lem}

Observe that $\sigma\in L^1(Y, \omega_Y^{\dim Y})$ since, by definition, $$\int _Y \sigma \omega_Y^{\dim Y}= \int_Y F_*  \omega_X^{\dim X}=\int_X  \omega_X^{\dim X}.$$

\begin{proof} 
Since we may blow up  $X$ ad libitum, we may assume $X$  is smooth. Assume first $Y$ is smooth. 
Then, we may assume $Y=(2\Delta)^{\dim Y}$,  $U=\Delta^{\dim Y}$. We may assume
$\omega_Y=\frac{i}{2} \sum_{j=1}^{\dim Y} dz_j\wedge d\bar z_j$ where $(z_j)$ are standard coordinates. 
Setting  $f_j=z_j\circ F$ we get
$$
(\frac{i}{2} \sum_{j=1}^{\dim Y} df_j\wedge d\bar f_j)^{\dim Y} \wedge \omega_X^{{\dim X}-{\dim Y}}=\mu \omega_X^{{\dim X}},
$$
where $\mu$ is a nonnegative real analytic function vanishing on $Crit(F)$. 
Consider $\gamma_1, \ldots \gamma_r\in \mathcal{O}(Y)$ holomorphic functions generating the ideal sheaf of the closed complex-analytic subspaces
 $F(Crit(F))$ near $U$ and define $$\lambda=\sum_j |\gamma_j\circ F|^2.$$
The zero set of the nonnegative real analytic function $\lambda$ contains the zero set of $\mu$ near $F^{-1}(U)$.
By Lojaziewicz's inequality, there exists constants $C,\alpha >0$ such that $\mu\ge C. \lambda^{\alpha}$. 

Thus if $y \not \in F(Crit(F))$, 
$$\sigma(y)=\int_{F^{-1}(y)} {\mu}^{-1} \omega_X^{\dim(X)-\dim(Y)}\le
C' (\sum_j |\gamma_j(y)|^2)^{-\alpha}. 
$$
It follows that
\begin{eqnarray*}
 \int_{U-F(Crit(F))} \sigma^{1+\epsilon} \omega_Y^{\dim Y} 
&\le &
C'\int_{U-F(Crit(F))}
  (\sum_j |\gamma_j|^2)^{-\alpha\epsilon} F_*\omega_X^{\dim X} \\
&=&  C'\int_{F^{-1}(U-F(Crit(F)))}
  (\sum_j |\gamma_j\circ F|^2)^{-\alpha\epsilon} \omega_X^{\dim X}\\
&<& +\infty, 
\end{eqnarray*}
if  $\epsilon$ is chosen so small that $\alpha\epsilon < c_{F^{-1} (\bar U)}(\log \lambda)$, where $c_{F^{-1} (\bar U)}(\log \lambda)$ denotes  the complex singularity exponent of the quasi-psh function 
$\log\lambda$ \cite{DK}. 

The case where $Y$ is singular follows from \cite[Lemma 3.2]{EGZ09}. 
\end{proof}

\subsection{The canonical volume form of a minimal abundant klt pair}

\subsubsection{Main construction}
Let $(X,\Delta)$ be a compact K\"ahler klt\footnote{See \cite{KM} for the definition. In particular, $\Delta$ is a $\Q$-Weil divisor whose coefficients lie in $]0,1[$.} pair and $h$ a smooth hermitian 
metric  on the $\Q$-line bundle  $L=(\mathcal{O}(N(K_X+\Delta))^{1/N}$, $N$ being the index of the 
$\Q$-Cartier divisor $K_X+\Delta$.

As in \cite[section 6]{EGZ09} define a  finite measure  on $X$ by 
$$v(h)=\frac{c_n\gamma\wedge\bar \gamma}{\| \gamma \|^2_{h}}$$ 
where $\gamma$ is a local multivalued non vanishing section of $L$ viewed as a multivalued top degree form with pole structure described by $\Delta$      and $c_n$ is
the unique complex number of modulus one making the expression positive.
The construction is independent of $\gamma$.

We assume furthermore that $(X, \Delta)$ is minimal and abundant :

\begin{defi}
A pair $(X, \Delta)$ is minimal and abundant
if  we may choose the first Chern form of $ h$ to be of the form $j^*\chi$ where: 
\begin{itemize}
 \item  $j: X \to J$ a surjective connected proper holomorphic map,  $J$ being a normal projective variety, 
 \item $\chi$ is a Hodge form on $J$, namely a K\"ahler class representing an ample $\Q$-Cartier divisor, 
 \item 
$\{ K_{X}+\Delta \}=j^*\{\chi\}$. 
\end{itemize}
\end{defi}

For $N\in \N^*$ such that $N.(K_X+\Delta)$ is Cartier, $j$ is the natural holomorphic map to
the canonical model of $X+\Delta$,
 $$
 J=\mathrm{Proj} \left(\bigoplus_{n\in \N} H^0(X, O_X(n.N. (K_X+\Delta) \right).
 $$

\begin{lem}\label{pushforward2} 
Notations from Lemma \ref{pushforward1}. Fix $\omega_J$  a K\"ahler form on $J$ . 
 There exists $\epsilon >0$ such that the  pushforward measure $w(h):=j_* v(h)$
 has $L^{1+\epsilon}$ density with respect to $\omega^{\dim J}_J$. 
\end{lem}

\begin{proof}
 This follows from Lemma \ref{pushforward1} in case $\Delta=\emptyset$ and $X$ has canonical singularities. In general,
one can find a generically finite  branched covering $f:X'\to X$ where $X'$ is a smooth K\"ahler orbifold  such that $f^*  v(j^*h)$ becomes 
a density with non-negative continuous coefficients, see \cite[section 3.4]{EGZ09}. It is straightforward 
to see that Lemma \ref{pushforward1} is also valid with $X$ replaced by an orbifold. Lemma  \ref{pushforward2} follows. 
\end{proof}

\begin{lem}\label{invar}
Let $h'=he^{-\phi}$, where $\phi\in C^{\infty}(J,\R)$. 
If $\psi, \psi'\in C^0(J,\R)$ are the unique solutions of the Monge-Amp\`ere equations on $J$
$$
 (\chi +dd^c \psi)^{\dim J}= e^{\psi} w(h), \  (\chi +dd^c \psi')^{\dim J}= e^{\psi'} w(h'),
$$
then $he^{-\psi}=h'e^{-\psi'}$. Hence 
$$
\theta_{can}=\chi +dd^c \psi \; \;
T_{can}=j^* \theta_{can} \in \{ K_X+\Delta\}
\;  \text{ and } \; v_{can}=e^{-\psi \circ j} v(j^*h)
$$
are all independent of $h$. 
\end{lem}

\begin{proof}   
 This follows from the identity $w(he^{-\phi})=e^{\phi}w(h)$. We refer the reader to \cite{EGZ09,CGZ13} for the 
existence and uniqueness
of continuous solutions to these Monge-Amp\`ere equations.
\end{proof}

\begin{defi} \label{stcurrent}
 The  adapted volume form $v_{can}=e^{-\psi \circ j} v(j^*h)$ 
is the Song-Tian canonical volume form on $(X,\Delta)$. 

Similarly we say that  $T_{can}{:=}T_{can}(X,\Delta)$ 
is the canonical semiflat representative of $K_{X, \Delta}=K_X+\Delta$ and that 
$\theta_{can}{:=} \theta_{can}(X,\Delta)$ 
is the twisted K\"ahler-Einstein metric attached to $(X,\Delta)$.
\end{defi}

\subsubsection{Variant of the main construction} \label{section:var}
For technical reasons, we have to introduce a variant of this construction. 

\begin{defi}
 Let $L$ be a holomorphic line bundle on a complex analytic space $M$. A hermitian metric with divisorial singularities on 
 $L$ is a singular hermitian metric $h$ such that $h$ is smooth on the complement of an effective Cartier divisor $E$  and there exists a smooth hermitian metric $\bar h$
 on $L\otimes \mathcal{O}(-E)$ coinciding with $h$ via the isomorphism $L\otimes\mathcal{O}(-E)|_{M\setminus E} \to L|_{M\setminus E}$ given by the 
 multiplication by the  section $s_E$ whose zero divisor is $E$. 
 
 The Cartier divisor $E$ is well defined and is called the polar divisor of $h$. One has $h=\bar h. h_E$ where $h_E$ is the singular hermitian metric 
 on $E$ whose curvature current is $E$. 
 
\end{defi}

Similarly a hermitian metric with divisorial singularities on a $\Q$ line bundle $L'$ is a hermitian metric with divisorial singularities on some line bundle of the form $L'^{\otimes N}, \ n\in \N^*$. 
The polar divisor is then an effective $\Q$-Cartier divisor.

Let $(X,\Delta)$ be a compact K\"ahler klt pair and $h$ a hermitian 
metric with divisorial singularities on the $\Q$-line bundle  $L=(\mathcal{O}(N(K_X+\Delta))^{1/N}$, $N$ being the index of the 
$\Q$-Cartier divisor $K_X+\Delta$.

Mimicking the previous construction, we define a measure with $L^1$ non negative density:
$$v(h)=\frac{c_n\gamma\wedge\bar \gamma}{\| \gamma \|^2_{h}}$$ 
where $\gamma$ is as above. The poles of $h$ give rise to zeroes of $v(h)$ on the polar divisor $E$ of $h$.

Assume furthermore that the first Chern form of $\bar h$ is $j^*\chi$ where: 
\begin{itemize}
 \item  $j: X \to J$ a surjective connected proper holomorphic map,  $J$ being a normal K\"ahler space, 
 \item 
$\{ K_{X}+\Delta-E\}=j^*\{\chi\}$ is the pull back from a smooth semi-K\"ahler class 
$\{ \chi\}$ with positive volume.
\end{itemize}

\begin{defi}
 Such quadruples  $(X,\Delta, E, j)$ will be called {\em  Iitaka models}. 
 \end{defi}
 
Lemmas \ref{pushforward2} and \ref{invar} still apply. 
Definition \ref{stcurrent} can then laid out as above. 
We shall need to emphasize dependence on the model, so we use the notations 
$$
\theta_{can}{:=} \theta_{can}(X,\Delta,E,j)
\; \; \text{ and }
\; \; 
\chi {:=}\chi_J.
$$

We remark that $E$ and $\Delta$ may have common components and that Iitaka models $(X,\Delta, E, j)$ and $(X, \Delta-\min(\Delta, E), E-\min(\Delta,E), j)$ are plainly different. 
However their $\theta_{can}$ agree. 

One could also work more generally with completely general singular hermitian metrics with  positive curvature current but we prefered to avoid excessive generality.

\subsubsection{Base change invariance}

The modified construction enjoys
a strong base change invariance. 

\begin{defi} \label{def:itmod} 
Let $(X,\Delta, E, j)$ be an Iitaka model. Let $\psi: \hat{J} \to J$ be a generically finite surjective morphism, $\hat{J}$ being normal K\"ahler. A base change of $(X,\Delta, E, j)$ along $\psi$
is an  Iitaka model $(\hat X, \hat{\Delta},\hat{E}, \hat{j})$ endowed with a morphism $\pi:\hat{X}\to X$ in such a way that: 
\begin{enumerate}
\item $\hat{j}: \hat{X}\to \hat{J}$ is a holomorphic connected fibration of compact K\"ahler spaces and $(\hat{X},\hat{\Delta})$ is klt
 \item We have a commutative square:
$$
\xymatrix@1{
\hat{X}\ar[r]^{\pi}\ar[d]_{\hat{j}} &X\ar[d]_{j} \\
\hat{J}\ar[r]^{\psi}& J\\
}
$$
\item There is a dense Zariski open set $J^0\subset J$ such that,  with $\hat{J}^{0}=\psi^{-1}(J^0)$,  the natural map from $\hat{X}^0= (\psi\circ \hat{j})^{-1}(J^0)$ 
to $X\times_J \hat{JÃ }^0$ is a proper bimeromorphic holomorphic mapping.
\item $K_{\hat{X}}+\hat{\Delta}-\hat{E}\sim_{\Q} \pi^*(K_X+\Delta -E)$.
\end{enumerate}
\end{defi}

The construction of $w(h)$ being base-change invariant thanks to condition (4) in definition \ref{def:itmod}, the following is simple:

\begin{lem}\label{lemma:bcinv} 
Let $(\hat{X},\hat{\Delta},\hat{E}, \hat{j})$ be a base change of $(X,\Delta,E,j)$ along $\psi$. Then
$\theta_{can}(\hat{X}, \hat{\Delta}, \hat{E}, \hat{j})=\psi^* \theta_{can} (X,\Delta,E, j)$ and
 $$ 
T_{can}(\hat{X}, \hat{\Delta}, \hat{E}, \hat{j})=\pi^* T_{can}(X,\Delta, E,j).
$$
\end{lem}

\begin{lem} 
Notations as in Definition \ref{def:itmod}. 
Assume $\psi$ is proper birational. Then for every log-resolution of singularities $\pi:\hat{X} \to X$ such that the meromorphic map $\hat{j}:=\psi^{-1}\circ j: \hat{X}\to \hat{J}$ 
is holomorphic there exists a unique $\Q$-Cartier divisor $\hat{\Delta}$ such that $(\hat{X},\hat{\Delta},\hat{E}, \hat{j})$ is a base change of $(X,\Delta,E, j)$ along $\psi$. 
\end{lem}

\begin{proof}
 By definition of klt singularities, we have 
 $$
 K_{\hat{X}}+\hat{\Delta}' \sim_{\Q}  \pi^* (K_X +\Delta) +\sum_i a_i E_i
 $$ 
 where $\hat{\Delta}'$ is the proper transform of $\Delta$ the $ E_i $ are the exceptional divisors and $a_i>-1$. 
 Then $\hat{\Delta}=\hat{\Delta}'+\sum_{-1<a_i<0} a_i E_i$ and $\hat{E}=\pi^*E+\sum_{a_i>0} a_i E_i$ satisfy our requirements.
 \end{proof}
 
\noindent  Observe that we can take $\hat{\Delta}=\emptyset$ if $\Delta=\emptyset$ and $X$ has canonical singularities. 
 
 \begin{coro} \label{goodmodel}
  If $(X,\Delta)$ is a nef abundant klt pair of non negative Kodaira dimension, then $(X,\Delta, \emptyset, j)$ has a base change $(\hat{X},\hat{\Delta}, \hat{E}, \hat{j})$ such that
  $\hat{J}$ is a log resolution of $J$, $\hat{X}$ is a
  log resolution of $X$, $\hat{j}:\hat{X}\to\hat{J}$ is smooth on the complement of a simple normal crossing divisor on $\hat{J}$. 
 \end{coro}
 
In order to have lighter notations, we will denote Iitaka models of the form $(X,\emptyset, E,j)$ by $(X,E,j)$ starting with next lemma. 

\begin{lem} \label{exist:bc}
Notations as in Definition \ref{def:itmod}. If $X$ is smooth and $\gamma$ is etale over $J^0$ there is a
 base change $(\hat{X}, \hat{E}, \hat{j})$ of $({X}, {E}, {j})$ along 
 $\gamma$ such that $\hat{X}$ is smooth and $\hat{X}^{0}=X\times_{J} J^0$. 
\end{lem}
\begin{proof}
The ramification divisor of $\hat{X} \to {X}$ needs to be absorbed in $\hat{E}$ for condition (4) in definition \ref{def:itmod} to hold. 
\end{proof}

\subsection{Regularity for smooth abundant minimal models}
 
Definition \ref{stcurrent} above applies, with $\Delta=\emptyset$ and $E=\emptyset$, when $X$ is a smooth abundant minimal model and $j: X \to J(X)$ is its Iitaka fibration.
Note that $J(X)$ is a projective normal variety although $X$ need not be projective. 

The critical locus $Crit(j)$ of $j$ is the closed algebraic subset of $J$ defined as 
the union of the singular locus of $J$ and the critical locus of $j^{-1}(J^{smooth}) \to J^{smooth}$.

\begin{prop}\label{smoothreg}
 $T_{can}$ is $C^{\infty}$  on $j^{-1}(J-Crit(j))$, $\theta_{can}$ is $C^{\infty}$  on $J-Crit(j).$
\end{prop}

\begin{proof}
{\bf {Step 1.}} Smoothness of $T_{can}$ on $j^{-1}(J-Crit(j))$ does not follow from \cite{EGZ09} but a proof is given in \cite{ST12}. 
We  give here an alternative proof, using covering tricks to simplify the geometric situation, to prepare for later generalizations.

 Choose a base change $(\hat{X}, \hat{E}, \hat{j})$ of $(X, \emptyset, j)$ as in Corollary \ref{goodmodel} 
 with the additional property that $J^0=J-Crit(J)$, $\hat{J}^0=J^0$ and thus $\hat{X}^0=X^0$ and we denote by $\hat{D}$ 
 the reduced simple normal crossing divisor $\hat{J}\setminus \hat{J}^0$.
We can find a big semiample $\Q$-Cartier divisor $\hat{A}$  on $\hat{J}$ such that  $K_{\hat{X}}-\hat{E} \sim_{\Q} (\hat{j})^* \hat{A}$. 

One has $\hat{E}=\sum_E a_E E$  with $a_E\in \Z_{> 0}$ and $\hat{j}(E)\subset \hat{D}$. 
One can also assume $\hat{J}^0\subset \mathrm{Amp}(\hat{A})$. 
In particular, the fibers  of $\hat{j}$ over $\hat{J}^0$ are smooth and have a torsion canonical bundle.

For $N$ divisible enough so that $N\hat{A}$ is Cartier and $N(K_ {\hat X}-\hat{E}) \sim N\hat{j}^*A$,
we have $\eta_N: \hat{j}^* O_{\hat{J}}(N\hat{A}) \to O_X (N.(K_{\hat{X}}-\hat{E}))$
the natural isomorphism of invertible sheaves. 
Let $\alpha$ be a local generator of $O_{\hat{J}}(N.\hat{A})$, let  $\| . \|$ be a smooth hermitian metric on $O_{\hat{J}}(N.\hat{A})$ whose curvature form is $N\chi_{\hat{J}}$
and denote by $\mu=\mu(h)$ the pushforward measure $$\mu=\hat{j}_*
\epsilon \frac{\eta_N( \hat{j}^*\alpha)^{1/N} \wedge 
\overline{ \eta_N( \hat{j}^*\alpha)^{1/N}}}{\|\alpha\|^{2/N}}.$$
Then $\theta_{can}(\hat{X}, \hat{E}, \hat{j})=\chi_{\hat{J}} + dd^c\phi$ satisfies: 
$$ (\chi_{\hat{J}} + dd^c\phi)^{\kappa(X)}= e^{\phi} \mu.
$$
One has locally on $\hat{J}$,  $\mu=d. e^{g}. |\beta|^2$ where $\beta$ is some local generator of the sheaf of
holomorphic canonical forms on $\hat{J}$,  $g$ is smooth and $d$ is the Hodge norm: 
$$
d(y)=\int_{X_y} |\frac{\eta_N( \hat{j}^*{\alpha} )^{\frac{1}{N}}} {\beta})|^2.
$$

We first treat the case $N=1$ in the next step.

\smallskip

\noindent  {\bf {Step 2.}} 
If we further assume that $\hat{j}:\hat{X}\to\hat{J}$ has unipotent local monodromy, we can apply  \cite[Theorem 1.1 (3)]{Kaw2}
to the effect that $d=\| \psi \|^2e^{-\phi^-}$ where $\phi^-$ is a local plurisubharmonic function with zero Lelong numbers  smooth outside $\hat{D}$ and $\psi$ is holomorphic. 
Indeed $\eta_1(\hat{j}^*\alpha)/\beta$ can be interpreted as a local section of the holomorphic line bundle $\hat{j}_*\omega_{\hat{X}|\hat{J}}$ and may vanish on $Crit(\hat{j})$
and $\phi^-$ the local potential of the 
Hodge norm has zero Lelong numbers. 

Then, we can  apply P\u{a}un's Laplacian estimate \cite{Paun08} to get that $\phi$ is smooth on $\hat{J}^0$, see also \cite[Theorem 10.1]{BBEGZ}. The exponential term $e^{\phi}$ in the Monge-Amp\`ere equation
does not appear in loc. cit. but this
does not affect the proof. 

In order to get rid of the unipotent monodromy condition, we use the classical covering trick -for which it is needed that $\hat{J}$ be projective: 

\begin{lem} \label{lem:ct} \cite{Kaw1}
There is a smooth ramified covering space $\gamma: Z \to \hat{J}$ whose ramification divisor $\hat{D}^+$
is a simple normal crossing divisor containing $\hat{D}$ such that the main component of $\hat{X}\times_{\hat{J}} Z$ has unipotent local monodromies.
Furthermore the allowable $\hat{J}\setminus \hat{D}^+$ form an open cover of $\hat{J}^0$ 
\footnote{One could take $\hat{D}^+=\hat{D}$ if one is ready to accept $Z$ to be an orbifold (a smooth Deligne-Mumford stack). }.
\end{lem}

We apply Lemma \ref{exist:bc} to produce a base change $(\hat{X}^*, \hat{E}^*, \hat{j}^*)$ of $(\hat{X}, \hat{E}, \hat{j})$ along 
 $\gamma$. The argument above yields smoothness of $\theta_{can} (\hat{X}^*, \hat{E}^*, \hat{j}^*)$ on the open set
 $Z\setminus \gamma^{-1}(\hat{D}^+)$ 
 and Lemma \ref{lemma:bcinv}
enables to descend to smoothness of $\theta_{can} (\hat{X}, \hat{E}, \hat{j})$ on $\hat{J}\setminus \hat{D}^+$. 
Lemma \ref{lem:ct} then yields that  $\theta_{can} (\hat{X}, \hat{E}, \hat{j})$ is smooth on $\hat{J}^0$. 

\smallskip

\noindent {\bf {Step 3.}}
To finish the proof, we need to  reduce to the $N=1$ case. Chose $N$ divisible enough and  $\sigma\in H^0(\hat{J}, O_{\hat{J}}(N \hat{A}))$ a  non zero global section. 
The standard  covering lemma enables to construct a base change  $\gamma$ etale over $\hat{J}\setminus \{ \sigma=0 \} \setminus G$
such that $\gamma^*\hat{A}$ is Cartier and  that $\gamma^*\sigma$ has a $N$-th root $\tau\in H^0(\hat{J}', O_{\hat{J}'}( A'))$ and the $\hat{J}\setminus \{ \sigma=0 \} \setminus G$ cover $\hat{J}\setminus \{ \sigma=0 \}$.
Applying lemma \ref{exist:bc}, smoothness of $\theta_{can}$ on $\hat{J}\setminus \{ \sigma=0 \}$ is reduced to prove smoothness on the non zero locus $\hat{J}^0$ of $\sigma$ in
the base-change invariant case where $\hat{A}$ is Cartier and has a non zero section $\sigma$. 

In that situation, the $N$-th power of the canonical sheaf $\omega_{\hat{X}^0}$ has $\eta_N(\sigma)$ as a 
non vanishing section and we can contruct an etale cyclic covering  $\hat{Y}^0 \to \hat{X}^0$ such that the  canonical sheaf $\omega_{\hat{Y}^0}$ is trivial
and the $N$-th power of a trivializing section $\eta_1^0(\sigma)$ is $\eta_N(\sigma)$. 
Since an etale covering of a quasi K\"ahler manifold is quasi K\"ahler by a slight variant of the Grauert-Remmert theorem, 
using Hironaka's theorem, we construct a Zariski open embedding $\hat{Y}^0 \to \hat{Y}$ such that $\hat{Y}$ is smooth and both $\xi: \hat{Y}\to X$ and $k: \hat{Y}\to \hat{J}$ are holomorphic. 

Denote the Stein factorization of $k$ by
$\hat{j}': \hat{Y}\to \hat{J}'$. By construction, $\hat{J}'$ is etale over the non zero locus of $\sigma$. 
Also by construction, we can set $\hat{\Delta}'=\emptyset$ $\hat{F}=K_{\hat{Y}/ X}$ the jacobian divisor of $\xi$ to the effect that 
$(\hat{Y}, \hat{F}, \hat{k})$ is a Iitaka model satisfying $K_{\hat{Y}}- \hat{F}'\sim \xi^*(K_X)$. 

This relation implies that
 the measure $\mu'$ on $\hat{J}'$ is a multiple (by the degree of the etale  mapping induced by $\xi$ on the general fibres of $j$)
 of the measure $\mu$ associated to a base change $(\hat{X}', \hat{E}', j')$  of $(X,\emptyset, j)$ along $\hat{J'}\to J$. 
 This implies that $$\theta_{can}(\hat{Y}, \hat{F}, \hat{j}')= \theta_{can} (\hat{X}', \hat{E}', j').$$

The $N=1$ argument applies to $(\hat{Y}, \hat{F}, \hat{j}')$ to yield smoothness of $\theta_{can}$ on the preimage of $\hat{J}^0$. 
 Indeed $\eta_1^0(\sigma)$ extends to a holomorphic canonical form $\eta_1(\sigma)$ on $\hat{Y}$ since its $N$-th power does. For the same reason, its zero divisor satisfies
 $(\eta_1)\ge \hat{F}$. Now, there is a unique meromorphic section $\eta$ of $ Hom_{O_{\hat{Y}}} (O_{\hat{Y}}(j^* A'), O_{\hat{Y}}(K_{\hat{Y}} -\hat{F}))$ sending 
 $(\hat{j}')^*\sigma$ to $\eta_1(\sigma)$. One has $\eta^{\otimes N}=\eta_N$. Hence $\eta$ is a holomorphic isomorphism. 
 The proof is complete. 
 \end{proof}

\subsection{Abundant minimal models with canonical singularities}

Let $X$ be an abundant projective variety with only canonical singularities and a nef canonical class  and let
$J: X \to J(X)$ be the Iitaka fibration which is 
induced by $|NK_X|$ for $N$ sufficiently divisible. 
then there is a ample $\Q$-Cartier divisor $A$ on
 $J(X)$ such that $K_X \sim_{\Q} J^*A$.
 
 \begin{prop}
The canonical current  $\theta_{can}(X,\emptyset,j)$ is smooth on a non-empty Zariski open subset of $J$. 
 \end{prop}

\begin{proof}
{\bf{Step 1.}} {We first reduce to the case when $K_X$ is a Cartier divisor.}
Choose $K_X$ a Weil divisor on $X$ such 
that $\omega_X=O_X(K_X)$. Applying \cite[pp. 361-362]{ypg}, we find 
$X^1\to X$  a cyclic covering with only canonical singularities such that $K_{X^1}=\pi^* K_{X}$ and
$K_{X^1}$ is Cartier. Then $K_{X^1}$ is semi-ample too.   
The Stein factorisation $j^1:X^1\to J^1$ 
is by construction the model of the Iitaka fibration
of $X^1$ induced by a sufficiently divisible
 pluricanonical system, in particular $J^1$ is normal. 
 As in the previous subsection we obtain $\theta_{can}(X^1, \emptyset, j^1)=p^*\theta_{can}(X, \emptyset, j)$ 
 where $p:J^1\to J$ is the natural finite mapping. Hence, 
 we can and will from now on assume that
$K_{X}$ is Cartier. 

\smallskip
 
 \noindent {\bf{Step 2.}} 
 {We now reduce to the case when $j_*\omega_{X}$ is a rank one coherent analytic sheaf having a non zero section. }
 This step is analogous to Step 4 in the proof of Proposition \ref{smoothreg}. We construct $J^0\subset J$ a Zariski open subset 
$\hat{J}^0 \to J^0$  finite etale, $\hat{X}^{0} \to X^0=j^{-1}(X^0)$ finite etale, such that $\omega_{\hat{X}^0}$ is trivial. We then construct a normal K\"ahler
compactification of $\hat{X}^{0}$ and use Hironaka's theorem to construct a compact connected K\"ahler manifold $\hat{X}'$ and a diagram of holomorphic connected mappings
$$
\xymatrix@1{
\hat{X}'\ar[r]^{\pi}\ar[d]_{\hat{j}'} &X\ar[d]_{j} \\
\hat {J}'\ar[r]^{\psi}& J\\
}
$$
with $\hat{J}'$ normal and $\psi$ finite such that  there is a log-resolution $\nu:X'\to X$ through which $\pi$ factorizes as $\pi=\nu \circ \rho$ with $\rho$ holomorphic
and $\hat{j}_*'\omega_{\hat{X}}$ is a non zero rank one coherent analytic sheaf.  The canonical divisors verify:
\begin{eqnarray*}
 K_{X'} & \sim_{\Q} & \nu^* K_{X} + \sum_{E \in Exc(\nu)} a_E E \\
 K_{\hat{X'}}& \sim_{\Q} &\rho^* K_{X'}+K_{\hat{X'}/X'} 
  \sim_{\Q}  \pi^* K_{X}+ K_{\hat{X'}/X'}+\nu^*  \sum_E a_E E
\end{eqnarray*}
with $K_{\hat{X'}/X'}\ge 0$ since both are smooth and all $a_E\ge 0$ since $X$ has canonical singularities. Therefore, setting:
$$
\hat{E}'=  K_{\hat{X'}/X'}+\nu^*  \sum_E a_E.E, 
$$
it follows that $(\hat{X}', \hat{E}', \hat{j'})$ is a Iitaka model, 
$\theta_{can}(\hat{X}', \hat{E}', \hat{j'})=\psi^* \theta_{can}(X,\emptyset,j)$ and we are reduced to proving smoothness
of $\theta_{can}(X,E,j)$ when $j_*\omega_{X}$ is a rank one coherent analytic sheaf having a non zero section. 

\smallskip

\noindent {\bf{Step 3.}} The situation can now be dealt with as in the proof of Proposition  \ref{smoothreg}. 
\end{proof}

 \begin{rem}
   The projective assumption on $X$ is only used in Step 1 and can probably be dropped. 
 \end{rem}

\subsection{A klt regularity result}  \label{sec:general}.
  
 In order to obtain a  generic regularity result for the twisted KE metric associated to a minimal abundant klt pair $(X,\Delta)$, we have to generalize
 our set-up to sub-klt pairs as was suggested to us by S. Druel.  The state of the art in the litterature regarding the covering lemmas
 we need seems to be \cite{A} and our write-up will follow this reference
  closely.

 Recall that a  pair $(X,B)$ consists of a normal complex projective variety $X$ endowed with a $\Q$-Weil divisor $B$ 
 such that $K_X+B$ is $\Q$-Cartier. 
 
 \begin{defi}\label{def:subklt}
 A pair $(X,B)$ is sub-klt (klt in loc. cit.)  whenever the Shokurov's discrepancy b-divisor ${\bf{A}}(X,B)$  satisfies
 $\lceil {\bf{A}}(X,B) \rceil\ge 0 $.  
\end{defi}

 For the convenience of the reader, we shall reproduce
the precise definition of ${\bf{A}}(X,B)$   \cite[p. 234]{A}.  Recall that a  b-divisor on a normal projective variety $X$ 
is the data $\mathbf{D}=\{ D_{X'} \}_{X'}$ of a $\Q$-Weil divisor on every (normal projective) birational model $X'$ of $X$ such that whenever the natural rational map
$\mu:  X''\to X'$ is a proper birational morphism, one has $\mu_* D_{X''}=D_X'$.  Every rational function $\phi$ (resp. top degree form $\omega$) on $X$ defines a 
linearily trivial b-divisor $\overline{(\phi)}:=\{(\phi_{X'})\}_{X'}$ (resp. a canonical b-divisor $\mathbf{K}=
\{ (\omega_{X'}) \}_{X'}$) . Every $\Q$-Cartier divisor $D$ on a birational model $X'$ gives rise to a $\Q$-Cartier b-divisor such that if $\mu:X''\to X'$ is a proper birational morphism
$\bar D_{X''}=\mu^* {D}_{X'}$. The round-up $\lceil \mathbf{D} \rceil$ of a b-divsor $\mathbf{D}$ is defined componentwise. Shokurov's discrepancy b-divisor for a pair $(X,B)$ 
is defined by the formula: 
$${\bf{A}}(X,B)=\mathbf{K}-\overline{K_X+B}. 
$$

 Definition \ref{def:subklt} means that 
 the usual inequality in the definition of klt singularities applies, in particular the multiplicities of $B$ are rational numbers strictly smaller than $1$.
 
 \begin{exa}
 If $(X,\Delta)$ is klt and $E$ is an
 effective $\Q$-divisor, then the pair $(X, \Delta-E)$ is sub-klt. 
 \end{exa}
 
 The reason sub-klt pairs are useful is the following:
 
 \begin{lem}\cite{SP} \label{subcrep}
 Let $(X,B)$ be sub-klt and $f:Y\to X$ be a generically finite morphism from a smooth projective variety $Y$. Then there exists a 
 unique $\Q$-divisor on $Y$,  $B_Y$,  such that:
 \begin{enumerate}
  \item $f^*(K_X+B)\sim_{\Q} K_Y+B_Y$
  \item $(Y,B_Y)$ is sub-klt.
 \end{enumerate}
 \end{lem}
 
  However $ B\ge 0$ does not imply $B_Z\ge 0$ when $f$ is not birational, even in dimension $1$, cf. also Lemma \ref{exist:bc}.
  
 To every $b$-divisor $\mathbf{D}$ we can associate a subsheaf $O_{X}(\mathbf{D})\subset \C(X)$ whose sections over a Zariski open subset $U\subset X$
 are the rational functions $\phi\in \C(X)$ such that $mult_{E}(\overline(\phi)+\mathbf{D})\ge 0$ whenever $E$ is an irreducible Weil divisor on some birational model $X'$
 such that its generic point maps to $U$.

 If $j:X\to Y$ is a proper regular fibration  \cite{A} defines  $j:(X,B)\to Y$ to be lc-trivial whenever
 \begin{enumerate}
  \item $(X,B)$ has sub-klt singularities over the generic point of $Y$
  \item $j_* O_X(\lceil {\bf{A}}(X,B) \rceil)$ has rank one
  \item there exists a positive integer, a rational function $\phi \in \C(X)$ and a $\Q$-Cartier divisor $D$ on $Y$ such that:
  $$K_X+B +\frac{1}{r} (\phi)=f^*D
  $$
  where $K_X$ is a canonical divisor. 
 \end{enumerate}

 Condition (2) should be seen as  controlling the negative part of $B$ :

\begin{exa}
 If $X$ smooth, $(X,B)$ is sub-klt and the support of $B$ is simple normal crossing $ O_X(\lceil {\bf{A}}(X,B) \rceil)= O_X(\lceil -B\rceil )$.
\end{exa}

Definition \ref{def:itmod} immediately applies to lc-trivial fibration and this coincides with the notion of base change of a lc-trivial fibration
used by \cite{A}. It follows from Lemma \ref{subcrep} that a lc-trivial fibration always admit a lc-trivial base change along any regular proper dominant morphism
from a smooth variety to its base. 

\begin{defi}
 Assume $j:(X,B) \to Y$ is an abundant lc-trivial fibration, meaning that  furthermore  $(X,B)$ is sub-klt and $D=A$ is big and semiample.
 
 If $f:X'\to X$ is a log-resolution of $(X,B)$ then
 setting $B_{X'}=\Delta-E$ (notations of lemma \ref{subcrep}) where $\Delta$ and $E$ are effective then $(X',\Delta, E, j\circ f)$ is an Iitaka model. 
 One defines
 $$
 \theta_{can}(X,B,j)=\theta_{can}(X',\Delta, E, j\circ f)
 $$
  which is independent of $X'$. 
\end{defi}

One does not need Condition (2) for this definition. It is however needed in the proof of
the following result :

\begin{theo}
 If $j:(X,B) \to Y$ is an abundant lc-trivial fibration, $\theta_{can}(X,B,j)$ is smooth on a non-empty Zariski open subset of $Y$. 
\end{theo}

\begin{proof} 
{\bf{Step 1.}} We will need Kawamata's canonical bundle formula hence we recall its main features.  
 For $p$ an irreducible Weil $\Q$-Cartier divisor on $Y$ define $b_P\in \Q$ by Kawamata's formula:
$$ 1-b_P=\min \{ t\in \R  \ | \ (X,B+tj^*P) \mathrm{\ is \ lc \ over \ the \ generic \ point \ of \ P} \}
$$
and define two $\Q$-Weil divisors on $Y$,  $B_Y= \sum b_P P$  and $M_Y$ by the canonical bundle formula:
 $$ K_X+B+\frac{1}{r} (\phi)= j^* (K_Y+B_Y+M_Y),
 $$
 where $K_Y$ is a canonical divisor. 
We then have the following result \cite{A}:

\begin{theo}  After a birational base change on $Y$, $(Y,B_Y)$ is subklt
and $(Y,M_Y)$ is nef. Both are invariant by further birational base change. 
\end{theo}

We can thus assume that $(Y,B_Y)$ is sub-klt and $(Y,M_Y)$ is nef. 

\smallskip

\noindent {\bf{Step 2.}}  
For $N$ divisible enough denote by 
$$
\eta_N: O_{X}(N j^*A) \to O_{X}(N. (K_X+B))
$$ 
the natural isomorphism. Denote by $\beta_N$ 
a local generator of $N.(K_Y+B_Y)$ and by $\alpha$ a local generator of $N.A$. Then $\frac{\eta_N(\alpha)}{\beta}$ can be identified with a
section of $N.M_Y$. When restricted to $F$,  $(\frac{\eta_N(\alpha)}{\beta})^{1/N} $ is a multisection of $K_F+B_F$.  
Furthermore the measure  $\mu$ such that $\theta_{can}=\chi_J+dd^c\phi$ and $(\chi_J+dd^c\phi)^{\dim Y}=e^{\phi} \mu$ is 
$$
\mu= e^{g} (\int_{X/Y} | (\frac{\eta_N(\alpha)}{\beta})^{1/N} |^2) ( \beta_N \overline{\beta_N})^{\frac{1}{N}},
$$
where $g$ is smooth.   

P\u{a}un's Laplacian estimate \cite{Paun08} reduces us to prove that 
$$
\int_{X/Y} | (\frac{\eta_N(\alpha)}{\beta})^{1/N} |^2=e^{\phi^+}e^{-\phi^-},
$$ 
possibly after a base change .
Here $\phi^-$ is a local plurisubharmonic function with zero Lelong numbers  smooth on the generic point of $J$
and $\phi^+$ is psh. Indeed since $B$ is subklt the density of $\mu$ with respect to Lebesgue measure would take the form 
$d=e^{\phi_B^++\phi^+-(\phi^-_B+\phi^-)}$ where
all four functions are quasi psh and 
$e^{-(\phi_B^{-}+\phi^-)} \in L^p$ with $p>1$ since $(Y,B_Y)$ is sub-klt so that $b_P<1$ in the formula $dd^c\phi^-_B=\sum_{b_P>0} b_P[P] + smooth$ . 

\smallskip

\noindent {\bf{Step 3.}} 
Let us recall Ambro's covering construction, namely \cite[Lemma 5.2]{A}.
After a birational base change ensuring that all the pairs under consideration are log-smooth and  the discriminants of $j$ and $h$ below are simple  normal crossing divisors,
see \cite[Properties (i)-(v), p. 245]{A}, we have the following geometric situation: 

Let $F$ be the generic fiber of $j$. 
Let $b$ be the smallest positive integer  $m\in \N^*$ such that $m(K_F+B_F)\sim O$. Let $\tilde{X}$ be the normalization of $X$ in 
$\C(X)[\sqrt[b]{\phi}]$. $\tilde X$ carries a Galois action of $G=\Z\slash b\Z$ so that $\tilde X/G=J$. Let $V$ be a smooth model of $\tilde X$ such that 
$h: V \to J$ is regular and the action of $G$ lifts to $V$. By Lemma \ref{subcrep} there is a sub-klt pair $(V,B_V)$ such that $k:(V,B_V) \to J$ satisfies all the properties of a 
lc-trivial fibration except Condition (2). 

Denoting by $\chi$ the character of $G$ sending $k \mod b$ to $e^{2\pi i k/b}$ and by
$(h_* \omega_{V/Y})^{\chi}$ the corresponding eigensubsheaf of $h_* \omega_{V/Y}$ we have, 
if $h$ is semistable in codimension one which can be assumed to be true, an isomorphism $O_J(M_J) \to (h_* \omega_{V/Y})^{\chi}$ - $M_J$ being effective. 
The Hodge metric on $h_* \omega_{V/Y}$ gives rise to a Hodge metric $\| . \|^2_{Hodge}$ on its summand $(h_* \omega_{V/Y})^{\chi}$. This is a singular hermitian 
metric on this line bundle
and by the method of \cite{Kaw2} it has a positive curvature current with zero Lelong numbers, after perhaps a further base change. Now it follows from
Ambro's construction that: 
$$
\int_{X/Y} | (\frac{\eta_N(\alpha)}{\beta})^{1/N} |^2=C.\|(\frac{\eta_N(\alpha)}{\beta})^{1/N}\|^2_{Hodge}
$$
where $C>0$ is a constant. This gives the required expression for the LHS of the Monge Amp\`ere equation. 
\end{proof}

 \begin{rem}
  A slightly less general result can be obtained using \cite[Theorem 1.3 (3)]{Kaw2}. 
 \end{rem}

\section{The parabolic viscosity approach} \label{sec:visc}

We study the complex degenerate parabolic complex Monge-Amp\` ere flows
\begin{equation} \label{krfglobal}
 e^{\dot{\f}_t + F (t,x,\f) } \mu (x) - (\omega_t +dd^c \f_t)^n = 0, 
\end{equation}
in $M_T=(0,T) \times M$, where $T \in ]0,+\infty]$ and
  \begin{itemize}
 \item $\overline {M}$ is a compact K\"ahler $n$-dimensional manifold with interior $M$ and boundary $\partial M := \overline{M} \setminus M$;
  \item $\omega_t (x)$ is a continuous family of semi-positive $(1,1)$-forms on $\overline{M}$, 
  \item $F (t,x,r)$ is continuous in $[0,T[ \times \overline{M} \times \R$ and non decreasing in $r$,
$(t,x,r) \mapsto F(t,x,r)$ is uniformly Lipschitz  in $r$ and 
$(t, x) \mapsto  F (t, x, 0)$ is uniformly bounded from above,
  \item $\mu (x)\geq 0$ is a bounded continuous volume form on $\overline{M}$,
   \item $\f : [0,T[ \times \overline{M} \rightarrow \R$ is the unknown function, with $\f_t: = \f (t,\cdot)$. 
   \end{itemize}

Given $\phi$ a continuous data on the the parabolic boundary $\partial_P M$ of $M$,
our aim here is to establish a global comparison principle between subsolutions and supersolutions to
the the Cauchy problem for the equation (\ref{krfglobal}).

For our applications, $M$ will be an open subset of a compact K\"ahler manifold $X$.
We shall always assume that
\begin{itemize}
\item there exists a big form $\theta$ such that $\omega_t \geq \theta$ for all $t$;
\item $M$ is  contained in the ample locus of $\theta$.
\end{itemize}
 We will also need to assume that $t \mapsto \omega_t$ is regular (see Definition \ref{def:regular}).

\subsection{Recap on viscosity concepts}

We recall some basic facts from the viscosity theory for complex Monge-Amp\`ere flows developed in \cite{EGZ16}.

\begin{defi}
A function $\f \in {\mathcal C}^{1,2}$ is a {\it classical subsolution} of (\ref{krfglobal}) if
for all $t \geq 0$ $x \mapsto \f_t(x)$ is $\omega_t$-psh and for all $(t,x) \in M_T$,
$$
(\omega_t+dd^c \f_t)^n \geq e^{\dot{\f}_t+F(t,x,\f_t(x))} \mu(x).
$$

A function $\f \in {\mathcal C}^{1,2}$ is a {\it classical supersolution} of (\ref{krfglobal}) if for all $(t,x) \in M_T$,
$$
(\omega_t+dd^c \f_t)_+^n \leq  e^{\dot{\f}_t+F(t,x,\f_t(x))} \mu(x)
$$

A {\it classical solution}   is
both a subsolution and a supersolution.
\end{defi}

\noindent Here $\theta_+(x)=\theta(x)$ if $\theta(x) \geq 0$ and  $\theta_+(x)=0$ otherwise.
{The problem is that classical solutions usually do not exist !}

 \begin{defi}
Given $u:M_T \rightarrow \R$ an u.s.c. bounded function
and $(t_0,x_0) \in M_T$, $q$ is a differential test form above for $u$ at $(t_0,x_0)$ if
\begin{itemize}
\item $q \in {\mathcal C}^{1,2}$ in a small neighborhood $V_0$ of $(t_0,x_0)$;
\item $u \leq q$ in $V_0$ and $u(t_0,x_0)=q(t_0,x_0)$.
\end{itemize}
\end{defi}

One defines similarly a differential test from below.

\begin{defi}
A bounded u.s.c.  function $u: M_T \rightarrow \R$  is a {\it viscosity subsolution} of (\ref{krfglobal})
if for all $(t_0,x_0) \in M_T$ and all differential test $q$ from above,
$$
(\omega_{t_0}(x_0)+dd^c q_{t_0}(x_0))^n \geq e^{\dot{q}_{t_0}(x_0)+F(t_0,x_0,q_{t_0}(x_0))} \mu(x_0).
$$
\end{defi}

\begin{defi}
A bounded  l.s.c. function $v: M_T \rightarrow \R$  is a {\it viscosity supersolution} of (\ref{krfglobal})
if for all $(t_0,x_0) \in M_T$ and all differential test $q$ from below,
$$
(\omega_{t_0}(x_0)+dd^c q_{t_0}(x_0))_+^n \leq e^{\dot{q}_{t_0}(x_0)+F(t_0,x_0,q_{t_0}(x_0))} \mu(x_0).
$$

A {\it viscosity solution} of (\ref{krfglobal}) is a continuous function which is both a viscosity subsolution and a viscosity supersolution.
\end{defi}

We thus went around the lack of regularity of the functions $u,v$ by using their differential tests (from below or/and above).
This definition is a special case of the general theory of \cite{CIL92} for viscosity solutions of general
degenerate elliptic/parabolic equations. The reader is referred to this survey article for the first principles 
of the theory.  

We refer the reader to \cite{EGZ15} for a study of viscosity sub/super-solutions to local complex Monge-Amp\`ere flows. 
We just recall here  some basic facts for the reader's convenience:
\begin{itemize}
\item Assume $u$ (resp. $v$) is ${\mathcal C}^{1,2}$-smooth. It is a viscosity subsolution (resp. supersolution) iff it is a classical subsolution (resp. supersolution).
\item If $u_1,u_2$ are viscosity subsolutions, then so is $\max(u_1,u_2)$.
\item If $(u_\alpha)_{\alpha \in A}$ is a  family of subsolutions which is locally uniformly bounded above, then
$$
\f:=\left( \sup \{ u_\alpha, \; \alpha \in A \} \right)^*
\text{ is a subsolution}.
$$
\item If $u$ is a subsolution of $(\ref{krfglobal})_{\mu}$
then it is also a subsolution of $(\ref{krfglobal})_{\nu}$ for all $0 \leq \nu \leq \mu$.
\item $u$ is a subsolution of $(\ref{krfglobal})_{0}$ if and only if  for all $t \geq 0$,  $x \mapsto \f_t(x)$ is $\omega_t$-plurisubharmonic.
\end{itemize}

\begin{itemize}
\item In the definition of differential tests, one can use functions that are merely Lipschitz in $t$
and $k$-convex (resp. $k$-concave) in $x$.  One then have to test the inequalities for all parabolic jets.
\item Set $A=2 Osc_X(u)$. If $u$ is a subsolution of $(\ref{krfglobal})_{fdV}$, then
$$
u^{\e}(t,x)=\sup\{ u(s,x)-\e^{-1}|t-s| ; \; |t-s|<A \e \}
$$
is Lipschitz in $t$,  decreases to $u$  as $\e$ decreases to zero, 
and is a subsolution of $(\ref{krfglobal})_{f_\e dV}$, where
$$
f_\e(t,x)=\inf \{ f(t,x); \; |t-s|<A\e \},
$$
replacing also $F$ by $F_\e$, the inf-convolution of $F$ in time.
\item One can also use sup-convolutions in space {\it locally}, considering 
$$
u^{\e}(t,x)=\sup\{ u(s,x)-\e^{-2} ||x-y||^2 ; \; |x-y|<A \e \}
$$
and obtain similar information.
\end{itemize}

A delicate point in the global setting 
is that there is no Lebesgue (i.e. translation invariant) measure.
Hence sup-convolutions 
are not so useful. We will first need  to localize before regularizing the objects of our study

\smallskip

Recall that if $\omega$ is a closed smooth $(1,1)$-form in $M$, then the complex Monge-Amp\`ere  measure $(\omega + dd^c \p)^n$ 
is well-defined {\it in the pluripotential sense} for all bounded $\omega$-psh functions $\p$ in $X$
by \cite{BT82}.
The  viscosity (sub)solutions of complex Monge-Amp\`ere equations can be interpreted in the pluripotential sense, as shown in \cite[Theorem 1.9]{EGZ11}. 

On the other hand, it is not yet clear how to interpret viscosity solutions of Complex Monge-Amp\`ere flows in terms of pluripotential theory. 
We note however the following useful lemma
which follows  from \cite[Theorem 1.9 and  Lemma 4.7]{EGZ11}:

\begin{lem}\label{lem:pluripotvisc}
 Assume $u\in C^0({M}_T, \R)$ is bounded and such that:
 \begin{itemize}
   \item the restriction $u_t$ of $u$ to $M_t :=  \{t \} \times M$ is $\omega_t$-psh,
  \item $u$ admits a continuous partial derivative $\partial_t u$ with respect to $t$,
  \item for every $ t\in ]0,T[$, the restriction $u_t$ of $u$ to $M_t$  satisfies  
  $$
  (\omega_t+dd^c u_t)^n \ge e^{\partial_t u+F(t,x,u)}\mu(t,x)
  $$ 
  in the pluripotential sense on $M_t$.
 \end{itemize}
 Then $u$ is a subsolution of (\ref{krfglobal}). 

 Similarly assume $v\in C^0({M}_T, \R)$ is bounded and satisfies  :
 \begin{itemize}
 \item the restriction $v_t$ of $v$ to $M_t :=  \{t \} \times M$ is $\omega_t$-psh,
  \item $v$ admits a continuous partial derivative $\partial_t v$ with respect to $t$,
  \item there exists a continuous function $w$ such that, for every $ t\in ]0,T[$, the restriction $v_t$ to $M_t$  satisfies 
  $$
  (\omega_t+dd^c v_t)^n \leq e^w \mu(t,x)
  $$ 
  in the pluripotential sense on $M_t$
  and ${\partial_t v_t+F(t,x,v_t)} \ge w$.
 \end{itemize}
Then $v$ is a supersolution of (\ref{krfglobal}). 
\end{lem}

There are two fundamental assumptions that we shall make on the semi-positive forms $\omega_t$.
We summarize them in the following definition:
 
\begin{defi} \label{def:regular}
We say that $t \mapsto \omega_t$ is regular if
 \begin{itemize}
\item $\exists \theta$ a semi-positive and big form s.t. $\theta \leq \omega_t$ for all t;
\item  $\exists \e \in {\mathcal C}^1$ with $\e(0)=0$ s.t. $(1-\e(t-s)) \omega_s  \leq \omega_t$.
\end{itemize}
\end{defi}

 We refer the reader to \cite[Lemma 2.7]{EGZ16} for a description of a wide class of forms
 $t \mapsto \omega_t$ which are regular.

\subsection{A refined local comparison principle}

We   establish here a local comparison principle which extends \cite[Theorem A]{EGZ15}.
We consider the following  complex Monge-Amp\`ere flows :
$$
e^{\dot u  + F (t,z,u)} \mu (t,z) + \beta (z)- (dd^c u)^n = 0, \, \, \text{in} \, \, \, D_T := ]0,T[ \times D, \leqno (MAF)_{F,\mu,\beta}
$$
where 

\begin{itemize}
\item $D$ is a complex manifold,
\item $\mu_t (x) = \mu (t,x) \geq 0$ is a continuous one parameter family of semi-positive volume forms on $D$,
\item  $\beta (x) \geq 0$ is a continuous semi-positive volume form on $D$.

\end{itemize}

\begin{lem} \label{lem:LCP1}
Assume $u : D_T \longrightarrow \R$ is a subsolution to $(MAF)_{F,\mu,\beta}$
and $v :   D_T \longrightarrow \R$ is a supersolution to  $(MAF)_{G,\nu,0}$, 
If  the function $u - v$ achieves a strict local maximum at some interior point $(t_0,z_0) \in D_T$ then :

\smallskip
 
 1. If $\beta > 0$ and  $\partial_t u  \leq \tau$, $\tau \geq 0$,  in a neighbourhood of $(t_0,z_0)$, then
\begin{equation} \label{eq:Fest}
e^{  F (t_0,0, u  (t_0,z_0))} \mu (t_0,z_0) +  e^{- \tau} \beta (z_0)
\leq e^{ G (t_0,z_0,v (t_0,z_0))} \nu (t_0,z_0),
\end{equation}
Thus $ \nu (t_0,z_0) > 0$ and  
$F (t_0,z_0,  u (t_0,z_0) <  G (t_0,z_0,v (t_0,z_0))$ when $\mu = \nu $.

\smallskip

2. If $\beta \geq 0$ and $\mu = \nu > 0$  then $F (t_0,z_0,  u (t_0,z_0)) \leq  G (t_0,z_0,v (t_0,z_0))$. 
\end{lem}

\begin{proof}
 The problem is local, so we can assume that $D = \B \Subset \C^n$ is the unit ball,   $z_0 = 0$ and  
$
M :=  \max_{\bar \B_T} (u - v) = u (t_0,0) - v (t_0,0).
$

\smallskip

{\it Step 1.}   We first assume that {\it  $\partial_t u$ and $\partial_t v$ are bounded} in a neighborhood of $(t_0,0)$. 
We use a classical doubling trick. For any $\e > 0,$ we define for $(t,x,y) \in [0,T[ \times \bar \B \times \bar \B$, 
$$
w_\e (t,x) := u (t,x) - \frac{\e}{T - t} - v (t,y) - \frac{1}{2 \e} \vert x - y\vert^2.
$$

By upper semi-continuity there exists $(t_\e, x_\e, y_\e) \in [0,T[ \times \bar \B \times \bar \B$ such that
$$
M_\e  = \max_{[0,T[ \times \bar \B \times \bar \B} w_\e
= u (t_\e,x_\e) - \frac{\e}{T - t_\e} - v (t_\e,y_\e) - \frac{1}{2 \e} \vert x_\e - y_\e\vert^2.
$$

It follows from \cite[Proposition 3.7]{CIL92} 
that $\vert x_\e - y_\e \vert^ 2 = o (\e)$ and that there is a subsequence $\e_j \to 0$ such that $ (t_{\e_j},x_{\e_j},y_{\e_j})$
 converges to $(\hat t,\hat x,\hat x) \in [0,T[ \times \overline{\B}^2$ where
 $(\hat t,\hat x)$ is a maximum point of $u - v$ on $\overline{\B}_T$
and 
\begin{equation} \label{eq:limmax}
\lim_{j\to \infty} {M}_{\e_j}= {M}.
\end{equation}

To simplify notation we set  for any $j \in \N$,
  $(t_j,x_j,y_j) = (t_{\e_j},x_{\e_j},y_{\e_j})$. Extracting and relabelling we may assume 
that $(u(t_j,x_j,y_j))_j$ and $(v(t_j,x_j,y_j))_j$ converge. 
By the semicontinuity of $u$ and $v$,
\begin{equation} \label{eq:semicont}
\lim_{j\to \infty} u(t_j,x_j,y_j) \le u(\hat{t}, \hat{x}), \  \lim_{j\to \infty} v(t_j,x_j,y_j)\ge v(\hat{t}, \hat{x}).
\end{equation}
On the other hand $(\ref{eq:limmax})$ implies that:
\begin{equation*} \lim_{j\to \infty} u(t_j,x_j,y_j) - \lim_{j\to \infty} v(t_j,x_j,y_j) 
 = u(\hat{t}, \hat{x})-v(\hat{t}, \hat{x}).
\end{equation*}
Together with $(\ref{eq:semicont})$, this yields
\begin{equation}\label{eq:limit} 
\lim_{j\to \infty} u(t_j,x_j)=u(\hat{t},\hat{x}), \ \lim_{j\to \infty} v(t_j,y_j)=v(\hat{t},\hat{x}). 
\end{equation}
and 
\begin{equation} \label{eq:Lineq2}
M = u (\hat t,\hat x) - v(\hat t,\hat x)\cdot
\end{equation}

From our assumptions it follows that $(\hat t,\hat x) = (t_0,0)$.
  Therefore we can assume that  for any $j \in \N$,
 $(t_j,x_j,y_j) = (t_{\e_j},x_{\e_j},y_{\e_j}) \in ]0,T[ \times \B^2$  and the sequence converges to $(t_0,0,0) \in ]0,T[ \times \B^2$. 
  
  Applying the parabolic Jensen-Ishii's maximum principle  
(the technical assumption being satisfied since $u$ and $v$ are locally Lipschitz in $t$) to the functions 
  $U (t,x) := u (t,x) - \frac{\e}{T - t}$,  
$ v$ and the penality function $\phi (t,x,y) :=  \frac{1}{2 \e} \vert x - y\vert^2$ 
for any fixed $\e = \e_j$, we  find  approximate parabolic second order jets
  $(\tau_j,p_j^{\pm},Q_j^{\pm}) \in \R \times \R^{2 n} \times \mathcal S_{2 n}$   such that 
  
 $$
 \left(\tau_j + \frac{\delta}{(T - t_j)^2} ,p_j^+,Q_j^{+}\right) \in \mathcal {\bar P}^{2,+} u (t_j,x_j), \ \ \left(\tau_j,p_j^-,Q_j^{-}\right) \in \mathcal {\bar P}^{2,-} v (t_j,y_j)
 $$
  with $ p_j^+ =  - p_j^- =   \frac{1}{\e_j} (x_j - y_j)$ and $Q_j^+ \leq Q_j^-$.

Let $H_j^{\pm}$ the hermitian $(1,1)$-part of $Q_j^{\pm}$. Then $H_j^+ \leq H_j^{-}$ for any $j \in N$ and from the viscosity inequalities satisfied by $u$ it follows that $H^+_j > 0$, hence $H^{-}_j > 0 $. 
Applying the parabolic viscosity differential inequalities for $u$ and $v$, we obtain for all $j \in \N$,
\begin{eqnarray*} 
\lefteqn{ e^{\tau_j + \frac{\e_j}{(T-t_j)^2 } + F (t_j, x_j, u (t_j,x_j))} \mu (t_j,x_j) + \beta (x_j) \leq   (dd^c H^+_j)^n } \\
&\leq &  (dd^c H_j^-)^n   \leq   e^{\tau_j + G (t_j, y_j,v (t_j,y_j))} \nu (t_j,y_j),  
\end{eqnarray*}
which implies that

\begin{equation} \label{eq:LipEst}
e^{ \frac{\e_j}{(T-t_j)^2 } + F (t_j, x_j,u (t_j,x_j))} \mu (t_j,x_j) + e^{- \tau} \beta (x_j) \leq e^{G (t_j,y_j,v (t_j,y_j))}\nu (t_j,y_j).
\end{equation}

Letting $j \to + \infty$ and using (\ref{eq:limit}) (\ref{eq:Lineq2}) , we obtain the
 inequality (\ref{eq:Fest}) in the case when $u$ and $v$ are locally uniformly Lipschitz in the variable $t$.

\smallskip

{\it Step 2.} To remove the assumption on $u$ and $v$, we use Lipschitz-regularization in the time variable.  
Let $u^k$ and $v_k$ denote the time-variable Lipschitz regularization of $u$ and $v$ respectively in a $[0,T] \times B_2$ 
(see (\cite[Lemma 2.5]{EGZ15}). 
Fix $T_0 \in ]0,T[$ with $T_0 < t_0 < T$ and consider for $k > 1$ large enough
$$
\overline{M}_k :=  \max_{[T_0,T]  \times \bar \B} \{ u^k (t,z) - v_k (t,z)  \}\cdot
$$
Observe that
$$ 
\lim_{k\to \infty}  \max_{[T_0,T] \times \bar \B} \left\{u^k (t,z) - v_k(t,z) \right\}= \max_{[T_0,T] \times \bar \B} \left\{ u (t,z)- v (t,z)\right\} = M.
$$
By upper-semicontinuity, there exists $(t_k,z_k) \in [T_0,T] \times \bar \B$ converging to some $(t',z') \in [T_0,T] \times \bar \B$ such that 
$
\overline{M}_k =  u^k (t_k,z_k) -v_k (t_k,z_k),
$
and
$$
 u (t',z') - v(t',z') \geq \lim_{k \to + \infty} \overline{M}_k =  M.
$$

Thus $ u (t',z') - v(t',z')  = M$, with $(t',z') \in [T_0,T] \times \bar \B$.
  Since  $ u - v$  attains a strict maximum in $[0,T] \times \bar \B$ at  $(t_0,0)$, it follows that $(t',z') = (t_0,0)$.

Extracting and relabelling  we can assume that for all $k$ the maximum $M_k$ is attained at an interior point $(t_k,z_k) \in ]T_0,T[ \times \B$, 
the sequence $(t_k,z_k)$ converges to $ (t_0,0)$, and 
$$
 \lim_{k \to + \infty} \{u^k (t_k,z_k) -v_k (t_k,z_k)\}  =  u (t_0,0) - v(t_0,0).
$$

Now $\limsup_{k \to + \infty} u^k (t_k,z_k)  \leq u (t_0,0)$ and $\liminf_{k \to + \infty} v_k (t_k,z_k) \geq v (t_0,0)$
since $u$ and $-v$ are upper semi-continuous. Thus  extracting again if necessary, we get
 $$
 \lim_{k \to + \infty} u^k (t_k,z_k)  = u (t_0,0), \, \, \, \lim_{k \to + \infty} v_k (t_k,z_k) = v (t_0,0).
 $$
It follows now from \cite[Lemma 2.5]{EGZ15} that for $k$ large enough, 
$$
e^{\partial_t u^k  + F_k (t,z,  u )} \mu_k \leq 
(dd^c u^k)^n, \, \, 
\text{ in } ]T_0,T[ \times B_2,
$$
where
$
F_k (t,z, r) := \inf_{\vert s - t\vert \leq A \slash k}  F (s,z, r), \, \, \mu_k (t,z):=  \inf_{\vert s - t\vert \leq A \slash k}  \mu (s,z),
$
 $A > 0$.
Similarly the functions $v_k$ satisfy 
$$
e^{\partial_t v_k  +  G^k (t,z,  v)} \nu^k  \geq 
(dd^c v_k)^n, \, \, 
\text{ in } ]0,T[ \times B_2,
$$
where
$
G^k (t,z, r) := \sup_{\vert s - t\vert \leq A \slash k} G (s,z, r), \, \, \nu^k (t,z) := \sup_{\vert s - t\vert \leq A \slash k} \nu (s,z).
$

Now $u^k$ and $ v_k$ satisfy all the requirements of {\it Step 1}.
 If $\beta \geq 0$ we  can conclude from (\ref{eq:LipEst}) that 
$$
e^{ F_k (t_k,z_k,  u (t_k,z_k))} \mu_k (t_k,z_k)  \leq  e^{G^k (t_k,z_k,v (t_k,z_k))} \nu^k (t_k,z_k).
$$

Letting $k \to \infty $ we get
$
e^{  F (t_0,0,  u (t_0,z_0))}  \mu (t_0,z_0) 
\leq e^{ G (t_0,0,v (t_0,0))} \nu (t_0,0),
$
hence
$$
F (t_0,0,  u (t_0,z_0)) \leq  G (t_0,0,v (t_0,0))
$$
if $\mu (t_0,z_0)  = \nu (t_0,z_0) > 0$.
\end{proof}

\begin{rem} 
 Lemma~\ref{lem:LCP1} is still true if  we merely assume that $u - v$ achieves its maximum on 
 $\bar \B_T :=[0,T] \times \bar \B$ at interior points only i.e. the set 
$$
K := \{ (t,z) \in [0,T] \times \bar \B ;  u (t,z) - v (t,z) = \max_{\bar \B_T} (u - v)\} \Subset ]0,T[ \times \B
$$
 is relatively compact in $\B_T$. This version implies the comparison principle proved in \cite[Theorem 4.2]{EGZ15}.
\end{rem}

As a consequence we obtain the following comparison principle for  
$$
e^{\dot u  + F (t,z,u)} \mu (t,z) - (dd^c u)^n = 0, \, \, \text{in} \, \, \, D_T := ]0,T[ \times D, \leqno (MAF)_{F,\mu}.
$$

\begin{coro} \label{coro:LCP2}
Let $u :  [0,T[ \times D \longrightarrow \R$ be a subsolution to  $(MAF)_{F,\mu}$
and  $v :  [0,T[ \times D \longrightarrow \R$  a supersolution to   $(MAF)_{F,\mu}$. 
 Assume that
 
 \smallskip
 
$(i)$  $u - v$ achieves a local maximum at some  point $(t_0,z_0) \in ]0,T[ \times D$,

 \smallskip

$(ii)$  the function $z \longmapsto u (t,z) - 2 c \vert z\vert^2$ is  plurisubharmonic near $z_0$, for some
 $c > 0$ and $t$ close to $t_0$.

 \smallskip

Then $  u (t_0,z_0) \leq v (t_0,z_0)$ and $u \leq v$ in a neighborhood of $(t_0,z_0)$.
\end{coro}

\begin{proof}
 We may assume that $z_0 = 0$ and there exists a constant $c > 0$ and $r > 0$ such that for $t - t_0 \leq r$, the function 
$z \longmapsto u (t,z) - 2 c \vert z\vert^2$ is plurisubharmonic in a neighborhood of the unit ball $\bar \B$.
Fix $\alpha \in ]0,1[$ and consider  
$$
u_\alpha (t,z) := u (t,z) - c \alpha \vert z\vert^2 - \alpha (t-t_0)^2, \; \; (t,z) \in [0,T[\times \bar D.
$$

The function $u_\alpha - v$ achieves a {\it strict local maximum} at $(t_0,0)$
and  for $\vert t - t_0\vert  \leq r$,
$$
z \mapsto u_\alpha (t,z) - (1-\alpha) u (t,z) - \alpha c \vert z\vert^2 - \alpha (t-t_0)^2 = \alpha (u (t,z) - 2 c \alpha \vert z\vert^2),
$$
is plurisubharmonic  in a neighborhood of the unit ball $\bar \B$, with 
$$
dd^c u_\alpha (t,\cdot) \geq (1-\alpha) \, dd^c u (t,\cdot) + \alpha \, c \, \beta_0,
$$
for $\vert t - t_0 \vert \leq r$,
where $\beta_0 := dd^c \vert z\vert^2$.
Thus $z \longmapsto u_\alpha (t,z)$ is strictly plurisubharmonic  for $\vert t - t_0 \vert \leq r$ and satistifes 
$$
 (dd^c u_\alpha)^n \geq (1 - \alpha)^n (dd^c u) + \alpha^n c^n \beta_0^n \geq e^{\partial_t u_\alpha + 2 \alpha (t-t_0) + F (t,z,u_\alpha) + \ln (1 - \alpha)} \mu,
$$
in the sense of viscosity in $]t_0-r , t_0 + r[ \times \B$.

\smallskip

Assume first that $\partial_t u \leq \tau$ in a neighbourhood of $(t_0,z_0)$, then $\partial_t u_\alpha \leq \tau + 2 \alpha t_0.$
Since $u_\alpha - v$ achieves a strict local maximum at $(t_0,0)$, we can apply Lemma~\ref{lem:LCP1} to conclude that
{\footnotesize
\begin{equation} \label{eq:Ineq}
 e^{F (t_0,0,u_\alpha (t_0,0)) + 2 \alpha (t-t_0) + \ln (1 - \alpha)} \mu (t_0,0) + \alpha^n c^n e^{\tau + 2 \alpha t_0} \beta_0^n\leq e^{G (t_0,0,v(t_0,0))} \nu (t_0,0).
\end{equation}
}
Letting $\alpha \to 0$ we obtain the required estimate.

\smallskip

To treat the general case, we approximate $u$ by a decreasing sequence $(u^k)$ of $k$-Lipschitz function in $t$.
 It is easy to check that $u^k$ satisfies the requirement of the first part and $\partial_t u^k \leq k$ . 
Arguing as in the proof of Lemma~\ref{lem:LCP1} we obtain   a sequence ($t_k,z_k) \in [0,T] \times \bar \B$ converging to $(t_0,0)$ such that $u^k - v_k$ achieves its maximum at $(t,_k,z_k)$ and 
$\lim (u^k (t_k,z_k) = u (t_0,0)$, $\lim_{k \to + \infty} v (t_k,z_k) = v (t_0,0)$ and from (\ref{eq:Ineq}) we conclude that
\begin{equation*}
 e^{F_k (t_k,z_k,u_\alpha (t_k,z_k)) + 2 \alpha (t-t_k) + \ln (1 - \alpha)} \mu (t_k,z_k) + \alpha^n c^n e^{k} \beta_0^n\leq e^{G (t_k,z_k,v(t_k,z_k))} \nu (t_k,z_k).
\end{equation*}

It follows that for all $k >>1$ and $\alpha \in ]0,1[$, 
\begin{equation*}
 e^{F_k (t_k,z_k,u^k_\alpha (t_k,z_k)) + 2 \alpha (t-t_k) + \ln (1 - \alpha)} \mu (t_k,z_k) \leq e^{G (t_k,z_k,v(t_k,z_k))} \nu (t_k,z_k).
\end{equation*}
This implies that $\nu (t_k,z_k) > 0$ and if $\mu = \nu$ we obtain  
$$
e^{F_k (t_k,z_k, u^k_\alpha (t_k,z_k)) + 2 \alpha (t-t_k) + \ln (1 - \alpha)}  \leq e^{G (t_k,z_k,v(t_k,z_k))},
$$
Letting $\alpha \to 0$ and $k \to + \infty$, we obtain the required inequality. 
\end{proof}

\subsection{Viscosity comparison principles for manifolds with boundary}
 
 Recall that the family $\omega_t$ satisfies   $\omega_t \geq \theta$ for all
 $t \in [0,T[$ and that $M$ is included in the ample locus of the cohomology class $\{\theta\}$.
Our purpose here is to establish the following version of the comparison principle:

\begin{theo} \label{thm:pcpabord}
Let $\f(t,x)$ (resp. $\p(t,x)$) be   a subsolution (resp. a supersolution) to the complex Monge-Amp\`ere flow (\ref{krfglobal})
in $M_T = ]0,T[ \times M$ which is bounded from above (resp. below).
Assume  that either $\partial_t \f$
 is locally bounded or the family $t \mapsto \omega_t$ is regular.
 Then for all $(t,x) \in [0,T[ \times M$,
$$
(\f-\p)(t,x) \, \, \leq  \, \, \sup_{\partial_P M} (\f^*-\p_*)_+,
$$
where 
$
\partial_P M := \{0\} \times \bar M) \cup ([0,T[ \times \partial M)
$
denotes  the parabolic boundary.
\end{theo}

Here the function $\f^*$ (resp. $\psi_*$) is the upper (resp. lower) semi-continuous extension of $\f$ (resp. $\psi$) to $\bar M$.
Note that this result contains the local comparison principle for domains of $\C^n$
established in \cite{EGZ15}.

\begin{proof} 
We  assume first that $\partial_t \f$ is  locally bounded from below in $]0,T[ \times M$. 
 Fix $\eta > 0$, $\e \in ]0,1[$ and  $\delta := \e (1 + \e)^{-1}$ so that $1- \delta =  (1 + \e)^{-1}$. 
 Consider
$$
\tilde \f (t,x) := (1 - \delta)\f(t,x)  + \delta \rho(x) - \frac{\eta}{T - t} - A (1 - \delta) t,
$$
where $A = A (\e) > 0$ will be chosen later,
$\rho \leq \f$ is a $\theta$-psh function satisfying $\theta + dd^c \rho \geq \beta$, $\beta$ is a 
K\"ahler form on $X$. Such a function exists since the cohomology class $\eta$ of $\theta$ is big. 
One can moreover impose $\rho < 0$ to be smooth in the ample locus $\Omega =\mathrm{Amp} \{ \theta \}$,
with analytic singularities,
and such that $\rho(x) \rightarrow -\infty$ as $x \rightarrow \partial \Omega$.

Since $ \tilde \f^* - \p_*$ is upper semi-continuous on $\bar M$,  the maximum of  $\tilde  \f^* - \p_*$ on $\bar M_T := [0,T[ \times \bar M$ is achieved at some point $(t_0,x_0) \in [0, T[ \times \bar M$, i.e. 
$$
 \sup_{ (t,x) \in  M_T} (\tilde \f^ (t,x) - \p_  (t,x)) = \tilde \f^* (t_0,x_0) - \p_*  (t_0,x_0).
$$

If $(t_0,x_0) \in \partial_P M_T$ we are done, 
so we  assume $(t_0,x_0) \in ]0,T[ \times M$, hence $\tilde \f^* (t_0,x_0) - \p_*  (t_0,x_0) = \tilde \f (t_0,x_0) - \p  (t_0,x_0)$. 
Observe that for all $t \in [0,T[$,   $\tilde \f$ satisfies  $\omega_t + dd^c \tilde \f_t \geq \delta \beta$ 
on $M$ and

$$
(\omega + dd^c \tilde \f)^n \geq (1 - \delta)^n (\omega + dd^c \f)^n + \delta^n \beta^n,
$$
in the viscosity sense in  $]0,T[ \times M$.
Now $\tilde \f \leq \f$ since $\rho < \f$ and 
$$ 
\partial_t \f = (1 + \e) \partial_t \tilde \f + \frac{\eta (1 + \e)}{(T - t)^2} + A \geq (1 + \e) \partial_t \tilde \f + \frac{\eta}{(T - t)^2} + A,
$$ 
hence
$$
(\omega_t + dd^c \tilde \f_t)^n \geq e^{(1 +\e) \partial_t \tilde \f + \frac{\eta}{(T -t)^2} + A + F (t,x, \tilde \f (t,x) + n \log (1 - \delta)}  \mu(x).
$$

We now   localize near $x_0$ and apply Corllary~\ref{coro:LCP2}. 
Choose  local coordinates  $z $ near $x_0$,
 identifying a closed neighborhood  of $x_0$ with
the closed complex ball $\bar B_2 := B(0,2) \subset \C^n$ of radius $2$, sending $x_0$ to the origin in $\C^n$. 
Let $ h = h_{\omega}  (x)$  be a smooth local potential for $\omega$ in $ B_2$,
i.e. $dd^c h_{\omega} = \omega $ in $ B_2$.

The function  $u (t,z):= \tilde \f (t,x) + h (t, x) $  is u.s.c. in $  [0,T[ \times B_2$ and for all
 $t \in [0,T[$, $ dd^c u_t = \omega_t + dd^c \tilde \f_t \geq \delta \beta$ on $B_2$, hence $u(t,\cdot) $ is strictly psh in $B_2$.

Our hypothesis insures that   $\partial_t \tilde \f \geq - C$ in $]t_0 - r,t_0 + r[ \times B_2$
if $0 < r < \min \{t_0, T- t_0\}$, hence $u$ satisfies  
$$
(dd^c u)^n \geq e^{\partial_t u   + \tilde F (t,z, u)}  \tilde \mu, \, \, 
\text{ in }  B_2,
$$
where $\tilde{\mu}=  z_* (\mu) > 0$  is a  positive continuous volume form on $ B_2$ and 
{\small
$$
\tilde F (t,z,r)  :=  \frac{\eta}{(T -t)^2} + F (t, x, r - h (t,x)) + A 
 - \partial_t h (t,x) - C \e + n \log (1 - \delta),
$$
}
is continuous in $[0,T[ \times B_2 \in \R$.
Similarly  the lower semi-continuous function 
$v  := \p \circ z^{- 1}  + h \circ z^{- 1} $
 satisfies the viscosity differential inequality
\begin{equation} \label{eq:supersol}
(dd^c v)_+^n \leq e^{ \partial v + \tilde G (t,\zeta, v (t,\zeta))} \tilde \nu , \, \, 
\text{ in }  B_2,
\end{equation} 
where $\tilde \nu :=  z_{*} (\nu) >  0$ 
is a  continuous volume form on $  B_2$ and $\tilde G (t,\zeta, r) := G (t, x, r) - h (t,x)) - \partial_t h (t,x)$ is continuous in $[0,T[ \times B_2 \times  \R$.

\smallskip

When $\mu = \nu$ and $F = G$ we apply Lemma~\ref{coro:LCP2} to conclude that
 \begin{equation*}
  \tilde F (t_0, 0, u (t_0,0))   \leq \tilde G (t_0,0, v (t_0,0)),
\end{equation*} 
hence   
 \begin{equation*}
 \frac{\eta}{(T -t_0)^2} + F (t_0, x_0, \tilde \f (t_0,x_0)) + A - C \e - n \log (1 + \e)    \leq F (t_0,x_0, \p (t_0,x_0)).
\end{equation*} 

Choosing $A = A (\e) := C \e + n \log (1 + \e)>0$  we conclude that 
$$
\frac{\eta}{(T -t_0)^2} + F (t_0, \tilde \f (t_0,x_0)) \leq F (t_0,\p(t_0,x_0)).
$$

We infer
$
\tilde \f(t_0,x_0 ) \leq \p(t_0,x_0)
$
since $F$ is non decreasing in the last variable.
As  $\tilde \f - \p$ achieves its supremum on $ M_T$ at $(t_0,x_0)$, we conclude that 
$$
(1 - \delta)\f(t,x)  + \delta \rho(x) - \frac{\eta}{T - t} -  A (1 - \delta) t  \leq  \p(t,x),
  $$
for all $(t,x) \in M_T$. Now $A  = C \e +  n \log (1 + \e) \to 0 $ as $\e \to 0$ and
 $  M \cap (\rho > - \infty\}  =\emptyset$. Letting $\e \to 0$ and $\eta \to 0$,  we 
 thus conclude that
$\f  \leq \psi $ in $M$.
 
\smallskip

We finally remove the condition on $\partial_t \f$  by approximating $\f$  by a decreasing  sequence 
$\f_k$ of $k$-Lipschitz functions in the time variable that satisfies an approximate viscosity sub-inequalities(see \cite{EGZ16}). 
This is  where the regularity assumption on $(t,x) \mapsto \omega_t(x)$ is used.
This permits to pass to the limit successfully and finaly get the result also in this case as in the proof of 
 \cite[Corollary 2.6]{EGZ16}. 
\end{proof}

 \begin{rem} 
  J.Streets pointed out to us that the proof of   \cite[Theorem~2.1]{EGZ16} is incomplete.
  We have been unable to fully repair the corresponding proof. 
  The proof of Theorem~\ref{thm:pcpabord} that we provide shows in particular that the comparison principle \cite[Theorem~2.1]{EGZ16}
  is at least valid in some Zariski open set (the ample locus of the cohomology class of $\theta$).
  
 This version is sufficient for all geometric applications.
Since the sub/super-solutions to the Cauchy problem for the parabolic complex Monge-Amp\`ere equation  constructed there are  continuous in $X_T$ (see \cite[Proposition 3.5]{EGZ16}), we can consider the upper envelope $\f$ of all the subsolutions to the corresponding Cauchy problem.
 
  By standard results from viscosity theory, $\f^*$ is a subsolution to the equation   while $\f_*$ is a supersolution. 
 Theorem~\ref{thm:pcpabord} thus shows that the statement in \cite[Corollary 3.3]{EGZ16} is valid on the regular part of $X$. Namely for all $(t,x) \in \R^+ \times X^{reg}$
 $$
 \f^* (t,x) - \f_* (t,x) \leq \max_{x \in X} (\f^* (0,x) - \f_* (0,x))_+.
 $$
 
Since the  sub/super-barriers constructed in \cite[Proposition 3.5]{EGZ16} and  \cite[Lemma 3.8]{EGZ16} are continuous, we deduce as in \cite[Theorem~3.11]{EGZ16} that $\f^* (t,x) \leq  \f_* (t,x)$ in $\R^+ \times X^{reg}$. Hence $\f^*  = \f_* = \f (t,x)$ in $\R^+ \times X^{reg}$. 
Thus $\f$ is the unique viscosity solution to the Cauchy problem.
 \end{rem}

\section{Long term behaviour of the normalized K\"ahler-Ricci flow} \label{sec:longterm}

We now study the long-term behavior of the normalized K\"ahler-Ricci flow
on an abundant minimal model $X$ of positive Kodaira dimension
and prove { Theorem D} assuming $X$ satisfies Conjecture C.

\subsection{The strategy}

\subsubsection{Normalization of the scalar parabolic equation}
Let $X$ be a compact $n$-dimensional K\"ahler variety with canonical singularities
such that $K_X$ is semi-ample. We let $j:X \rightarrow J$ be an Itaka morphism and
$A$ an ample line bundle on $J$ such that $K_X=j^*A$.
We also let $\kappa=\kappa(X)>0$ denote the Kodaira dimension of $X$.

Denote by $\omega_t$ the normalized K\"ahler-Ricci flow on $X$ starting at $\omega_0$ \cite{EGZ16}. 
The cohomology class $\alpha_t$ of $\omega_t$ satisfies
$$
\alpha_t=e^{-t} \alpha_0+(1-e^{-t}) c_1(K_X),
$$
where $\alpha_0$ is the initial K\"ahler class. Pick $\omega_J \in \{A\}$ a K\"ahler form
and set 
$$
\chi:=j^* \omega_J.
$$
We also pick $\omega_0 \in \alpha_0$ a K\"ahler form representing $\alpha_0$ and set
$$
\theta_t:=\chi+e^{-t}(\omega_0-\chi)=e^{-t} \omega_0+(1-e^{-t}) \chi.
$$
Note that this family (as well as its pull-back on any log-resolution) satisfies all the requirements
from section \ref{sec:visc}.  

We fix $v(h)$ a canonical volume form 
on $X$ such that 
$
\chi=-{\rm{Ric}}(v(h)) \, \, \, \text{in} \, \, \, X.
$
We assume without loss of generality that $v(h)$ and $\alpha_0$ are normalized so that
$$
\int_X v(h)=\int_X \chi^{\kappa} \wedge \omega_0^{n-\kappa} 
=c_1(K_X)^{\kappa} \cdot  \alpha_0^{n-\kappa}=1.
$$

It is classical that  the normalized K\"ahler-Ricci flow is equivalent 
the following parabolic complex Monge-Amp\`ere flow of potentials,
\begin{equation} \label{eq:flotfinal}
\frac{(\theta_t+dd^c \f_t)^n}{C_n^{\kappa} e^{-(n-\kappa)t}}=  e^{\partial_t \f+\f_t} v(h),
\end{equation}
starting from an initial smooth 
(or possibly continuous) 
K\"ahler potential $\f_0 \in PSH(X,\omega_0)$.
We have normalized here both sides so that the volume of the left hand side converges to
$1$ as $t \rightarrow +\infty$. Here $C_n^k$ denotes the binomial coefficient
$
C_n^k=\left( \begin{array}{c} n \\ k \end{array} \right).
$ 
The initial value problem $\phi(-, 0)=\varphi_0$ for this flow admits a unique  viscosity solution which is (locally) upper bounded in $\R^+ \times X$ and continuous in $\R^+ \times X^{reg}$
 and $\omega_t=\theta_t+dd^c\varphi$ is the normalized K\"ahler-Ricci flow starting at $\omega_0+dd^c\varphi_0$.
We remark, although we will not use this fact, that $\varphi$ can be shown to be smooth on $]0,\infty[ \times  X^{reg}$.

\subsubsection{Canonical  and semi-flat currents}
 The {\it canonical current} is the positive closed $(1,1)$-current $T_{can}=\chi+dd^c \f_{\infty}$ on $X$.
Its potential $\f_{\infty}$ is continuous in $X$ and  smooth in a Zariski open set $X \setminus D$, where $D$ denotes a divisor on $X$
that both contains the singular fibers of $j$ and the singular points of $X$.

Thus for all compact subset $K \subset X \setminus D$, there exists $C_K>0$ such that
\begin{equation} \label{eq:c2bdd}
0 \leq \chi+dd^c \f_{\infty} \leq C_K \, \chi.
\end{equation}

\smallskip

Using Conjecture C we will assume that the {\it semi-flat current} $\omega_{SF}=\omega_0+dd^c \rho$ defined in the introduction is also smooth on $X \setminus D$. 
As already mentioned, we will assume, without loss of generality, that $\int_{X_y} \omega_0^{n-\kappa}=1$. 

\begin{lem} On $X \setminus D$, we have:
$$
(\chi +dd^c \f_{\infty})^{\kappa} \wedge \omega_{SF}^{n-\kappa}=e^{\f_{\infty}} v(h).
$$
\end{lem}

\begin{proof} 
 If $\gamma$ is a local section of $j_*\omega_{X|J}^{\otimes N}$ defined over a Zariski open subset $J^0\subset J$ over which 
$\f_{\infty}$ is smooth and
$\eta_N:j^*j_*\omega_{X/J}^{\otimes N} \to \omega_{X/J}^{\otimes N}$ is the natural sheaf morphism -which is an isomorphism-, 
we have
$$
\omega_{SF}^{n-\kappa}|_{X_y} =\frac{|\eta_N(\gamma)|^{2/N}}{\|\gamma^{1/N}\|^2_{Hodge}}
$$ 
on $j^{-1}(J^0)$. 

The left hand side can be rewritten as 
$
(\chi +dd^c \f_{\infty})^{\kappa}\frac{|\eta_N(\gamma)|^{2/N}}{\|\gamma^{1/N}\|^2_{Hodge}}
$
 and it is well-defined as a  smooth hermitian metric of the $\Q$-line bundle ${\mathcal O}_{j^{-1}(J^0)}(K_X)$. 
Its Ricci curvature is zero along the fibers 
of $j$. The same being true for the right hand side,  there exists a function $c:J^0 \to \R$ such that 
$$
(\chi +dd^c \f_{\infty})^{\kappa} \wedge \omega_{SF}^{n-\kappa}=e^c.e^{\f_{\infty}} v(h).
$$
 Taking the integral along the fiber of $j$ we obtain $(\chi +dd^c \f_{\infty})^{\kappa}=e^{c}e^{\f_{\infty}}v(h)$ hence $c=0$ 
by Definition \ref{stcurrent}. 
\end{proof}

\subsubsection{The plan}
The idea of the proof of  { Theorem D}  is standard.
We would like to construct a subsolution $u(t,x)$ and a supersolution $v(t,x)$ of the
flow (\ref{eq:flotfinal}) such that  for all $x \in X^{reg}$,
$$
\f_{\infty}(x) \leq \lim_{t \rightarrow +\infty} u(t,x)
$$
and
$$
\f_{\infty}(x) \geq \lim_{t \rightarrow +\infty} v(t,x).
$$
 It would then follow from the comparison principle Theorem~\ref{thm:pcpabord}  that for all $x \in X^{reg}$,
$$
\f_{\infty}(x) = \lim_{t \rightarrow +\infty} \f(t,x).
$$

Unfortunately, we  will have to slightly modify this strategy since we  can only construct approximate sub/super-solutions  
in a large subdomain of $X^{reg}$. This is enough thanks to Theorem~\ref{thm:pcpabord}.

\subsection{Uniform bounds}

\subsubsection{Uniform upper-bound}

We first observe that the solution $\f(t,x)$ of the flow is uniformly bounded from above
on $\R^+ \times X$:

\begin{lem} \label{lem:unifbdd}
There exists $C>0$ such that for all $t>0$ and $x \in X$,
$$
\f(t,x) \leq C.
$$
\end{lem}

\begin{proof}
Set $v(t,x)=C$, where $C>\sup_X \f_0$. Then $v(0,x) \geq \f_0(x)$
and 
$$
(\theta_t+dd^c v_t)_+^n \leq (\chi+e^{-t} \omega_0)^n
=\sum_{j=0}^{\kappa} C_n^j \chi^j \wedge \omega_0^{n-j} e^{-(n-j)t}.
$$

Each term $e^{-(n-j)t}$, $0 \leq j \leq \kappa$, is bounded from above by 
$e^{-(n-\kappa)t}$, while the terms $\chi^j \wedge \omega_0^{n-j}$ are bounded
from above by $C' \omega_0^n \leq C'' v(h)$, as explained in the proof of \cite[Lemma 6.4]{EGZ09}.
We infer
$$
\frac{(\theta_t+dd^c v_t)_+^n}{C_n^k e^{-(n-\kappa)t}}  \leq  
e^{A} v(h) =     e^{\partial_t{v}+v-C} e^{A} v(h)  \leq e^{\partial_t{v}+v} v(h) 
$$
if $C \geq A$. 

It follows therefore from Lemma \ref{lem:pluripotvisc} that $v$ is a global supersolution of the Cauchy problem for (\ref{eq:flotfinal}).
The comparison principle Theorem~\ref{thm:pcpabord} thus yields the uniform upper-bound
$
\f(t,x) \leq v(t,x)=C.
$

\smallskip

We also note the following alternative proof of independent interest, as it only requires 
$X$ to have log terminal singularities. Set
$$
I(t):=\int_X \f_t \, v(h).
$$

The function $\f_t$ is $\theta_t$-psh, hence $(\omega_0+\chi)$-psh
since $0 \leq \theta_t \leq \omega_0+\chi$ for  $t \geq 0$.
Since the singularities of $X$ are log terminal we have $PSH(X,\omega_0+\chi) \subset L^1(v(h))$, 
hence it follows from \cite[Proposition 1.7]{GZ05} that 
$$
\sup_X \f_t \leq I(t)+C,
$$
with $C$ independent of $t$. Jensen's inequality yields
\begin{eqnarray*}
I'(t)+I(t)=\int_X (\partial_t \f_t+\f_t) v(h) &\leq &\log \int_X e^{\partial_t \f_t+\f_t} v(h) \\
&=&\log \left( \frac{\int_X \theta_t^n}{C_n^{\kappa} e^{-(n-\kappa)t}} \right) \leq B.
\end{eqnarray*}
We infer $I(t) \leq I(0)+B$ hence $\f(t,x) \leq I(0)+B+C$.

This argument only requires that $\f$ is a subsolution, but we have implicitly used that $t \mapsto \f_t$ is Lipschitz
to differentiate $I$. To circumvent this difficulty, we can first replace $\f$ by its supconvolution in time $\f^\e$
(which is a subsolution of an approximate flow)
and then proceed as above.
\end{proof}

\subsubsection{Uniform lower-bound}

A first order argument is required to establish a uniform lower bound on the solution $\f(t,x)$:

\begin{lem} \label{lem:unifbdd2} 
There exists $C>0$ such that for all $t>0$ and $x \in X$,
$$
\f(t,x) \geq -C.
$$
\end{lem}

\begin{proof} The proof uses basically the same idea as the analogous result in \cite{ST12}. 
We let $\pi: Z \rightarrow X$ denote a desingularization of $X$ and pull-back the flow to $Z$,
obtaining
$$
\frac{(\pi^*\theta_t+dd^c \f_t \circ \pi)^n}{C_n^{\kappa} e^{-(n-\kappa)t}}=  e^{\partial_t \f \circ \pi +\f_t \circ \pi} \pi^*v(h).
$$

We let $\omega_Z$ denote a K\"ahler form on $Z$. 
The measure $\pi^* v(h)$ can be written as
$F \omega_Z^n$, where $F \geq 0$ is a smooth function (here we use the hypothesis that $X$ has only canonical singularities).
We let $\f^{\e}$ denote the approximating flows, solutions of the smooth parabolic approximating flows,
$$
\frac{(\pi^*\theta_t+\e \omega_Z+dd^c \f_t^{\e})^n}{C_n^{\kappa} e^{-(n-\kappa)t}}=  e^{\partial_t \f^\e   +\f_t^\e} (F+\e) \omega_Z^n,
$$
with smooth initial data $\f_0^\e$, which decreases uniformly towards $\f_0$.

It follows from standard viscosity theory \cite[section 6]{CIL92}
that the functions $\f^\e$ are $(\omega_0+\e \omega_Z)$-psh and  uniformly converge to $\f \circ \pi$, as $\e \rightarrow 0^+$.

It therefore suffices to establish a uniform bound from below for $\f^\e$, which is independent of $\e$. We set
$$
H_\e(t,x)=(e^t-1) \partial_t \f_t^\e(x)-\f_t^\e(x)-h(t),
$$
where $h$ will be chosen below. We let the reader 
check that
\begin{eqnarray*}
\left( \frac{\partial}{\partial t}-\Delta_t \right)(H_\e)
&=& n-h'(t)-(e^t-1) C'(t)-\rm{Tr}_t(\pi^* \omega_0+\e \omega_Z) \\
& \leq & n-h'(t)-(e^t-1) C'(t),
\end{eqnarray*}
where $C(t)=(n-\kappa)t-\log C_n^{\kappa}$ and $\Delta_t$, $\rm{Tr}_t$ denote the
Laplacian and Trace operators with respect to the K\"ahler form $\pi^*\theta_t+\e \omega_Z+dd^c \f_t^{\e}$.
 We thus choose 
$$
h(t)=(1+\kappa) t +(n-\kappa) e^{t}, 
$$
so that $\left( \frac{\partial}{\partial t}-\Delta_t \right)(H_\e) \leq -1<0$.

It follows that $H_\e$ attains its maximum along $(t=0)$. Lemma \ref{lem:unifbdd}  yields
$$
(e^t-1) \partial_t \f_t^\e(x) \leq C_1+(n-\kappa) e^t+(1+\kappa) t
$$
hence for all $t \geq 1$ and $x \in X$,
$$
\partial_t \f_t^\e(x) \leq  C_2.
$$
Lemma \ref{lem:unifbdd} again and the main result of \cite{EGZ08,DP10} insure that
$$
0 \leq \sup_Z \f_t^\e -\inf_Z \f_t^\e \leq C_3,
$$
with $C_3$ independent of $t,\e$.
 
 It remains to uniformly bound from below $\sup_Z \f_t^\e$.
 Observe that for $0 < \e<1$,
 $$
1 \leq \int_Z \frac{ (\pi^* \theta_t +\e \omega_Z)^n}{C_n^{\kappa} e^{-(n-\kappa)t}} =\int_Z e^{\partial_t \f^\e+\f_t^\e}  (F+\e) \omega_Z^n
\leq \int_Z e^{C_2+\f_t^\e}  (F+1) \omega_Z^n.
 $$
 
 Setting $\mu=e^{C_2}(F+1) \omega_Z^n$, we infer
 $$
 1 \leq \int_Z e^{\f_t^\e}  d \mu \leq \mu(Z) e^{\sup_Z \f_t^\e},
 $$
 hence $\sup_Z \f_t^\e$ is uniformly bounded below and the proof is complete.
\end{proof}

\subsection{Construction of subsolutions} \label{sec:subsol}

\subsubsection{Optimal lower bound}

Recall that the current $\omega_{SF}=\omega_0+dd^c \rho$ 
is such that  $\omega_y={\omega_0}_{|X_y}+dd^c \rho_{|X_y}$ is the unique 
Ricci-flat current on $X_y$ which is cohomologous to ${\omega_0}_{|X_y}$.
The potential $\rho_{|X_y}$ is  normalized so that $\int_{X_y} \rho_{|X_y} \, {\omega_0}_{|X_y}^{n-\kappa}=0$.

We assume in this section that $\rho$ is $\omega_0$-psh and obtain an efficient lower bound for $\f$:

\begin{prop} \label{pro:subsol}
Assume $\rho$ is $\omega_0$-psh.
Fix $C \geq \sup_X (\rho-\f_0)$. The function
$$
(t,x) \in \R^+ \times X \mapsto u(t,x):=(1-e^{-t}) \f_{\infty}(x)+e^{-t} \rho(x)-Ce^{-t}+h(t) \in \R
$$
is a subsolution to the Cauchy problem for (\ref{eq:flotfinal}).
\end{prop}

Here $h$ denotes the solution of the ordinary differential equation (ODE) (\ref{eq:subsol2}) below.

\begin{proof}
 Observe first that $u(0,x)=\rho(x)-C+h(0)=\rho(x)-C  \leq \f_0$,
by our choice of $C$.
We now check that $u$ is a subsolution to the equation, i.e. satisfies
$$
(\theta_t+dd^c u_t)^n \geq C_n^{\kappa} e^{-(n-\kappa)t} e^{\partial_t u+u_t} v(h)
$$
in the viscosity sense. Recall that
$(\chi +dd^c \f_{\infty})^{\kappa} \wedge \omega_{SF}^{n-\kappa}=e^{\f_{\infty}} v(h)$. Since
$\theta_t=(1-e^{-t}) \chi+e^{-t} \omega_0$,
we note that $u_t$ is $\theta_t$-psh with
\begin{eqnarray*}
(\theta_t+dd^c u_t)^n &\geq&  C_n^{\kappa} e^{-(n-\kappa)t} \left(1-e^{-t} \right)^{\kappa} 
(\chi +dd^c \f_{\infty})^{\kappa} \wedge \omega_{SF}^{n-\kappa} \\
&=& C_n^{\kappa} e^{-(n-\kappa)t} e^{\partial_t u+u_t} v(h)
\end{eqnarray*}
if the function $h$ solves solves
$$
h'(t)+h(t)=\kappa \ln \left(1-e^{-t} \right),
$$
with  $h(0)=0$.
The conclusion follows from Lemma \ref{lem:pluripotvisc}.
\end{proof}

We let the reader verify the following elementary result:

\begin{lem} \label{lem:subsol}
Let $h:\R^+ \rightarrow \R^+$ be the smooth solution of the ODE
\begin{eqnarray} \label{eq:subsol2}
h'(t)+h(t)=\kappa \ln(1-e^{-t})
\; \; \text{ with } \; \;
h(0)=0.
\end{eqnarray}
There exists $C>0$ such that for all $t\geq 0$,
$-C (t+1) e^{-t} \leq h(t) \leq 0$.
\end{lem}

When $\rho$ is $\omega_0$-psh, we therefore have a precise asymptotic bound from below for
$\f$, with exponential speed:

\begin{coro}
If $\rho$ is $\omega_0$-psh, then for all $t,x$,
$$
-C'(t+1)e^{-t}+e^{-t} \rho(x) \leq \f(t,x)-\f_{\infty}(x).
$$
\end{coro}

\subsubsection{Approximate subsolutions}

In this section we provide a construction of approximate subsolutions with good asymptotic behavior,
assuming $\rho$ is smooth in $X \setminus D$ (rather than $\omega_0$-psh),  where
$D=(s=0)=j^*D'$ is a divisor in $X$.

We fix $h$ a smooth hermitian metric 
of the line bundle $L_D=j^*L_{D'}$ 
$L_D$ and normalize $h$ so that $|s|_h \leq 1$ on $X$. 
We can assume without loss of generality (up to rescaling) that the curvature form $\Theta_h$ of $h$
 is dominated by $\chi$.
 The Poincar\'e-Lelong equation thus yields
\begin{equation}  
\chi+dd^c \log |s|_h \geq \Theta_h +dd^c \log |s|_h =[D] \geq 0,
\end{equation}
where $[D]$ denotes the current of integration along $D$.
We set
$$
V_{r}(D):=\{ x \in X \;  ;  \; |s(x)|_h < r \}
$$
for $r>0$ fixed, and $\Omega_r:=X^{reg} \setminus V_{r}(D)$.
We  assume that the semi-flat current $\omega_{SF}=\omega_0+dd^c \rho$
is smooth in $X \setminus D$, thus there exists $C_r>0$ such that 
$$
\omega_0+dd^c \rho \leq C_{r} \, \omega_0
\text{ in } X \setminus V_{r}(D).
$$

We can now proceed with the construction of a subsolution in $\Omega_r = X^{reg} \setminus V_{r}(D)$.
For technical reasons, we  provide a supersolution in
$[T_0,+\infty[ \times \Omega_r$, where the stopping time
$T_0$ depends on a parameter $\e>0$ and $r=r(\e)$ :

\begin{prop}  \label{pro:subsol2}
 Fix $\e>0$. 
 
 1) Fix $r=r_\e>0$ s.t.
{ $(1-\e) \sup_X \f_{\infty}+\frac{\e}{2} \log r_\e \leq \sup_{\R^+ \times X} \f(t,x)$.}

\smallskip

2) Fix $T_0 \geq 0$ so large that $e^{-T_0} \rho_{| \partial \Omega_r}+\frac{\e}{2} \log r \leq 0$.

\smallskip

3) Fix $C>1$ s.t. $(1-\e) \sup_X \f_{\infty}+e^{-T_0} \sup_{\Omega_r} \rho-Ce^{-T_0} \leq \f_{T_0}$.

\smallskip

\noindent Then  
$$
(t,x) \mapsto u_\e(t,x)=[1-e^{-t}-\e] \f_{\infty}(x)+e^{-t} \rho(x)+\e \log |s|_h(x)-C e^{-t}+h(t)
$$
is a subsolution to the Cauchy problem for (\ref{eq:flotfinal}) in $ M_r :=[T_0,+\infty[ \times \Omega_r$.

Here  $h:\R \rightarrow \R^+$ denotes unique the solution of the ODE
$$
h'(t)+h(t)=\ln \left[(1-e^{-t}-\e )^{\kappa}-e^{B-t} \right]
$$
with initial condition $h(0)=0$, where $B>0$ is specified in (\ref{eq:subgal}) below.
\end{prop}

\begin{proof}
We first note that $u_\e \leq \f$ on the parabolic boundary
$$
\partial_p M_r= \{T_0\} \times \Omega_r \bigcup [T_0,+\infty[ \times \partial \Omega_r.
$$
Indeed $h \leq 0$ and $\log|s|_h \leq 0$ hence 
$$
u_\e(T_0,x) \leq [1-\e] \f_{\infty}(x)+e^{-T_0} \rho-C e^{-T_0} \leq \f_{T_0}(x),
$$ 
as follows from our choice of $C$ (3), while for $x \in \partial \Omega_r=\{x \in X , \, |s(x)|_h=r \}$
the inequality $u_\e \leq \f$ follows from 
conditions (1) and (2).
Let us stress that the uniform upper-bound Lemma \ref{lem:unifbdd2} is used both in conditions (1) and (3).

We now check that $u_\e$ is a subsolution of the equation in $M_r$. 
Observe that $[1-e^{-t}-\e] \f_{\infty}+\e \log |s|_h$ is $(1-e^{-t})\chi$-psh, thus
\begin{eqnarray*}
\lefteqn{
(\omega_t+dd^c u_\e)^n
= \sum_{j=0}^\kappa C_n^j e^{-(n-j)t} (1-e^{-t})^j \chi_{\f_{\infty}}^j \wedge {\omega_{SF}}^{n-j} }\\
& & \geq  C_n^{\kappa} e^{-(n-\kappa)t} \left\{ [1-e^{-t}-\e]^{\kappa}-C_r e^{-t} \right\} e^{\f_{\infty}} v(h) \\
&& \geq  C_n^{\kappa} e^{-(n-\kappa)t}  \left\{ [1-e^{-t}-\e]^{\kappa}-C_r e^{-t} \right\} e^{(1-\e) \f_{\infty}+\e \inf_X \f_{\infty}+ \e \log |s|_h } v(h) \\
&& =  C_n^{\kappa} e^{-(n-\kappa)t} e^{\partial_t u_\e+u_\e} v(h)
\end{eqnarray*}
if $h$ satisfies
\begin{equation} \label{eq:subgal}
h(t)+h'(t)=\e \inf_X \f_{\infty}+\log \left\{ [1-e^{-t}-\e]^{\kappa}- e^{B-t} \right\},
\end{equation}
with $C_r=e^B$ and initial condition $h(0)=0$.
\end{proof}

 \begin{coro}
For all $x \in X^{reg} \setminus D$,
$$
\f_{\infty}(x) \leq \liminf_{t \rightarrow +\infty} \f(t,x).
$$
\end{coro}

\begin{proof}
We let the reader   check that 
$$
 \e \inf_X \f_{\infty}+\kappa \log[1-\e] \leq B  \e \leq \liminf_{t \rightarrow +\infty} h(t).
$$
It follows therefore from the Comparison Theorem \ref{thm:pcpabord} that for $x \in \Omega_r$,
$$
[1-\e] \f_{\infty}(x)+\e \log |s|_h(x)+B \e \leq \liminf_{t \rightarrow +\infty} u_\e(t,x) \leq \liminf_{t \rightarrow +\infty} \f(t,x). 
$$

We let $r \rightarrow 0^+$, so that this inequality holds for all $x \in X^{reg} \setminus D$.
Since $\e>0$ is arbitrary, we infer  $\f_{\infty}(x) \leq \liminf_{t \rightarrow +\infty} \f(t,x)$
for all $x \in X \setminus D$.
\end{proof}

\subsection{Construction of supersolutions} \label{sec:supersol}

\subsubsection{Holomorphic submersions}

We first treat the restrictive case when the Iitaka fibration is a holomorphic submersion.
This allows to explain the main idea in a simple setting.

\begin{prop} \label{pro:supsolkappa1}
Assume $j:X \rightarrow J$ is a smooth holomorphic submersion.
Fix $C \geq sup_X (\f_0-\rho)$. The function
$$
(t,x) \in \R^+ \times X \mapsto v(t,x):=(1-e^{-t}) \f_{\infty}(x)+e^{-t}\rho(x)+Ce^{-t}+g(t) \in \R
$$
is a supersolution to the Cauchy problem for (\ref{eq:flotfinal}).
\end{prop}

Here $g$ denote the solution of the ordinary differential equation   (\ref{eq:supsol}).

\begin{proof} 
The choice of $C$ insures that $v(0,x) \geq \f_0(x)$.
We now check that $v$ is a supersolution to the equation, i.e. satisfies
$$
(\theta_t+dd^c v_t)_+^n \leq C_n^{\kappa} e^{-(n-\kappa)t} e^{\partial_t v+v_t} v(h)
$$
in the viscosity sense, by using Lemma \ref{lem:pluripotvisc}.

The situation is more involved than for subsolutions, as we cannot get rid of the mixed  terms 
$(\chi +dd^c \f_{\infty})^{j} \wedge (\omega_{SF})_+^{n-j}$, when $j<\kappa$.
We take advantage of the assumption that $j:X \rightarrow J$ is a smooth holomorphic submersion, as in this case the extra terms
are smooth: for $j<\kappa$, we get
$ \chi_{\f_{\infty}}^{j} \wedge \omega_{SF}^{n-j} \leq n e^{\f_{\infty}+B'} v(h)$ for an appropriate choice of $B'$, hence
\begin{eqnarray*}
(\theta_t+dd^c v_t)_+^n &\leq&  C_n^{\kappa} e^{-(n-\kappa)t} \left\{
\chi_{\f_{\infty}}^{\kappa} \wedge \omega_{SF}^{n-\kappa} + 
c_n e^{-t}  \sum_{j=0}^{\kappa-1}  \chi_{\f_{\infty}}^{j} \wedge \omega_{SF}^{n-j}  \right\}\\
& \leq &  C_n^{\kappa} e^{-(n-\kappa)t}   e^{\f_{\infty}} v(h) \left\{ 1+e^{B-t} \right\}    \\
&=& C_n^{\kappa} e^{-(n-\kappa)t}  e^{\partial_t u+u_t} v(h),
\end{eqnarray*}
if the function $g$ solves solves
$
g'(t)+g(t)=\ln \left(1+e^{B-t} \right)
$
with  $g(0)=0$.
\end{proof}

\begin{lem} \label{lem:supsolkappa1}
Let $g:\R^+ \rightarrow \R^+$ be the smooth solution of the ODE
\begin{eqnarray} \label{eq:supsol}
g'(t)+g(t)=\kappa \ln(1+e^{B-t})
\; \; \text{ with } \; \;
g(0)=0.
\end{eqnarray}
There exists $C>0$ such that $0 \leq g(t) \leq C (t+1) e^{-t}$ for all $t \geq 0$.
\end{lem}

It follows from Propositions \ref{pro:subsol} and \ref{pro:supsolkappa1},
Lemmata \ref{lem:subsol} and \ref{lem:supsolkappa1},
and an application of the comparison principle Theorem~\ref{thm:pcpabord}
that:

\begin{coro}  
When $j:X \rightarrow J$ is a smooth holomorphic submersion,
there exists $C>0$ such that for all $x \in X$ and all $t \geq 0$,
$$
|\f_t(x)-\f_{\infty}(x)| \leq C [1+t] e^{-t}
$$
\end{coro}

The flow therefore uniformly deforms the initial continuous potential $\f_0$
 towards $\f_{\infty}$, at an exponential speed.

\subsubsection{The general case}

We now come back to the general setting of an abundant minimal model $X$ with 
canonical singularities, and let $j:X \rightarrow J$ denote the Iitaka fibration onto the canonical
model of $X$.

We are now going to use the fact that $\f_{\infty}$ is smooth in $X \setminus D$, where
$D=(s=0)=f^{-1}D'$ is a divisor in $X$ (see section \ref{sec:st}). We fix $h$ a smooth hermitian metric
of the line bundle $L_D$ and normalize $h$ so that $|s|_h \leq 1$ on $X$. 
Since $D=f^{-1}D'$, we observe that there exists $A>0$ such that the curvature of $h$
satisfies $\Theta_h \leq A \chi$. The Poincar\'e-Lelong equation thus yields
\begin{equation} \label{eq:curvbdd}
-dd^c \log|s|_h=-[D]+\Theta_h \leq A \chi,
\end{equation}
where $[D]$ denotes the current of integration along $D$.
We also set
$$
V_{r}(D):=\{ x \in X \;  ;  \; |s(x)|_h < r \}
$$
for $r>0$ fixed. The regularity assumption insures that
$$
\chi+dd^c \f_{\infty} \leq C_{r} \, \chi
\text{ in } X \setminus V_{r}(D),
$$
for some constant $C_r$ large enough. We set $\Omega_r:=X^{reg} \setminus V_{r}(D)$.

We moreover assume that the semi-flat current $\omega_{SF}=\omega_0+dd^c \rho$
is smooth in $X \setminus D$, thus there exists $C_r'>0$ such that 
$$
(\omega_0+dd^c \rho)_+ \leq C'_{r} \, \omega_0
\text{ in } X \setminus V_{r}(D).
$$

Now $\chi^j \wedge \omega_0^{n-j}  \leq C'' v(h)$
on $X$, so we obtain the following upper-bound  in $X \setminus V_{r}(D)$:
\begin{equation} \label{eq:supergeneral}
\sum_{j=0}^{\kappa-1} C_n^j (\chi+dd^c \f_{\infty})^j \wedge (\omega_0+dd^c \rho)_+^{n-j}
\leq C_r'' \, C_n^{\kappa} v(h).
\end{equation}

We can now proceed with the construction of a supersolution.
For technical reasons, we  provide a supersolution in
$[T_0,+\infty[ \times X^{reg} \setminus V_{r}(D)$, where the stopping time
$T_0$ depends on a parameter $\e>0$ and $r=r(\e)$ :

\begin{prop} \label{pro:supsolkappageneral}
 Fix $\e>0$. 
 
 1) Fix $r=r_\e>0$ so small that 
{\small $[1+\e A] \inf_X \f_{\infty}-\frac{\e}{2} \log r_\e \geq \sup_{\R^+ \times X} \f(t,x)$.}

\smallskip

2) Fix $T_0=T_0(\e) \geq 0$ so large that $e^{-T_0} \rho_{| \partial \Omega_r}-\frac{\e}{2} \log r \geq 0$.

\smallskip

3) Fix $C=C(\e)>1$ s.t. $[1+\e A] \inf_X \f_{\infty}+e^{-T_0} \inf_{\Omega_r} \rho+Ce^{-T_0} \geq \f_{T_0}$.

\smallskip

\noindent Then  
$$
(t,x) \mapsto v_\e(t,x)=[1+\e A] \f_{\infty}(x)+e^{-t} \rho(x)-\e \log |s|_h(x)+C e^{-t}+g(t)
$$
is a supersolution to the Cauchy problem for (\ref{eq:flotfinal}) in $ M_r=]T_0,+\infty[ \times \Omega_r$.
\end{prop}

Here  $g:\R \rightarrow \R^+$ denotes unique the solution of the ODE
$$
g'(t)+g(t)=\ln \left[(1+\e A)^{\kappa}+e^{B-t} \right]-A \, \e \inf_X \f_{\infty}
$$
with initial condition $g(0)=0$, where $B>0$ is specified in (\ref{eq:br}) below.

\begin{proof}
We first note that $v_\e \geq \f$ on the parabolic boundary
$$
\partial_p M_r= \{T_0\} \times \Omega_r \bigcup [T_0,+\infty[ \times \partial \Omega_r.
$$
Indeed $g \geq 0$ and $-\log|s|_h \geq 0$ hence 
$$
v_\e(T_0,x) \geq [1+\e A] \f_{\infty}(x)+e^{-T_0} \rho+ C e^{-T_0} \geq \f_{T_0}(x),
$$ 
as follows from our choice of $C$ (3), while for $x \in \partial \Omega_r=\{x \in X , \, |s(x)|_h=r \}$
the inequality $v_\e \geq \f$ follows from 
conditions (1) and (2).
Let us stress that the uniform upper-bound in Lemma \ref{lem:unifbdd} is used both in conditions (1) and (3).

We now check that $v_\e$ is a supersolution of the equation in $M_r$. 
We set $\chi_{\f_{\infty}}:=\chi+dd^c \f_{\infty}$.
Recall that 
$
-\e dd^c \log |s|_h \leq A \e \chi,
$
hence
$$
(\theta_t+dd^c v_\e)_+ \leq (1+A \e) \chi+e^{-t} (\omega_{SF})_+.
$$
We infer
$$
(\theta_t+dd^c v_\e)_+^n \leq  
 \sum_{j=0}^{\kappa} C_n^j \, (1+\e A)^j \, e^{-(n-j)t} \, \chi_{\f_{\infty}}^j \wedge (\omega_{SF})_+^{n-j} .
$$
Using that $\f_{\infty}$ and $\rho$ are smooth in $ \Omega_r$, we get
\begin{equation} \label{eq:br}
 (\theta_t+dd^c v_\e)_+^n
\leq C_n^{\kappa} e^{-(n-\kappa)t} e^{\f_{\infty}} \, v(h)
\left\{(1+\e A)^{\kappa} +e^{B-t} \right\}
\end{equation}
for some constant $B>0$ that depends on all our parameters.
Now 
\begin{eqnarray*}
\partial_t v_\e+v_\e &=& (1+\e A) \f_{\infty}-\e \log |s|_h+g'+g \\
& \geq & \f_{\infty}+A \, \e  \inf_X \f_{\infty} +g'+g,
\end{eqnarray*}
 since we have normalized $h$ so that $\log |s|_h \leq 0$.
 
 Thus $(\theta_t+dd^c v_\e)_+^n \leq e^{\partial_t v_\e+v_\e } v(h)$ and
it follows finally from Lemma \ref{lem:pluripotvisc} that $v_\e$ is a viscosity supersolution of the equation in 
$[T_0,+\infty[  \times \Omega_r$.
\end{proof}

\begin{coro}
For all $x \in X^{reg} \setminus D$,
$$
\limsup_{t \rightarrow +\infty} \f(t,x) \leq \f_{\infty}(x).
$$
\end{coro}

\begin{proof}
We let the reader   check that $\limsup_{t \rightarrow +\infty} g(t) \leq M \e$,
with $M=A \inf_X \f_{\infty}+\kappa$,
 thus for $x \in \Omega_r$,
$$
\limsup_{t \rightarrow +\infty} \f(t,x) \leq  \limsup_{t \rightarrow +\infty} v_\e(t,x) \leq [1+\e A] \f_{\infty}(x)-\e \log |s|_h(x)+M\e.
$$

We let $r \rightarrow 0^+$, so that this inequality holds for all $x \in X^{reg} \setminus D$.
Since $\e>0$ is arbitrary, we infer $\limsup_{t \rightarrow +\infty} \f(t,x) \leq \f_{\infty}(x)$
for all $x \in X^{reg} \setminus D$.
\end{proof}

\subsection{Proof of Theorem D}

It follows from Lemmata \ref{lem:unifbdd} and \ref{lem:unifbdd2} that $(\f_t)$ is uniformly bounded as $t \rightarrow +\infty$.
We let $\sigma$ denote a cluster point (for the $L^1$-topology), which is thus bounded and $\chi$-psh in $X$.
We can assume without loss of generality that the convergence  holds almost everywhere.
Previous Corollary shows that for a.e. $x \in X^{reg} \setminus D$, $\sigma(x) \leq \f_{\infty}(x)$, hence everywhere in $X$, the two functions being quasi-psh in $X$, 
while the comparison principle Theorem~\ref{thm:pcpabord}  and Propositions \ref{pro:subsol}, \ref{pro:subsol2} yield
$\sigma(x) \geq \f_{\infty}(x)$ in $X$.
Therefore $\sigma=\f_{\infty}$ and $\f_t$ locally uniformly converges to $\f_{\infty}$ in $X \setminus D$, as $t \rightarrow +\infty$.
This actually implies that $\f_t$  converges to $\f_{\infty}$ in capacity as $t \to + \infty$.
The proof of Theorem D is complete.


\begin{thebibliography}{99}

 
  \bibitem[Amb04]{A} F. ~Ambro: {\it  Shokurov's boundary property} J. Diff. Geom. {\bf 67} (2004), 229--255. 

\bibitem[BT82]{BT82} E.~Bedford, B.A.~Taylor: { \it A new capacity for plurisubharmonic functions. } Acta Math.  {\bf 149}  (1982), 1--40.
 
\bibitem[BBEGZ]{BBEGZ} R.~Berman, S.~Boucksom, P.Eyssidieux, V.~ Guedj, A.~Zeriahi: {\it  K\"ahler-Einstein metrics and the K\"ahler-Ricci flow on log Fano varieties}. Preprint,  arXiv:1111.7158.
  
\bibitem[BCHM10]{BCHM} C.~Birkar, P.~Cascini, C.~Hacon, J.~McKernan: {\it Existence of minimal models for varieties of log general type.}  J. Amer. Math. Soc. {\bf 23} (2010), no. 2, 405-468.
 
 
 \bibitem[Cao85]{Cao85} H.D. ~Cao: {\it  Deformation of K\"ahler metrics to K\"ahler-Einstein metrics on compact K\"ahler manifolds. }
 Invent. Math. {\bf 81} (1985) 359-372. 
 
 \bibitem[CPH16]{CPH} F. ~Campana, A. ~H\"oring, T. ~Peternell { \it Abundance for K\"ahler threefolds}.
 Ann. Sci. Éc. Norm. Supér. (4) 49 (2016), no. 4, 971–1025. 
 

 \bibitem[Choi15]{Choi15} Y.-J.Choi:{\it  Semi-positivity of fiberwise Ricci-flat metrics on Calabi-Yau fibrations.} Preprint   arXiv:1508.00323v4
 
 \bibitem[CGZ13]{CGZ13}  D. Coman, V.~ Guedj, A.~ Zeriahi: {\it  Extension of plurisubharmonic functions with growth control.}  J. Reine Angew. Math. 676 (2013), 33–49.
  
 \bibitem[CIL92]{CIL92} M.~Crandall, H.~Ishii, P.-L.Lions : { \it User's guide to viscosity solutions of second order partial differential equations} Bull. Amer. Math. Soc. {\bf 27} (1992), 1-67. 

\bibitem[DK01]{DK} J.-P.Demailly, J.Koll\'ar : { \it Semi-continuity of complex singularity exponents and K\"ahler-Einstein metrics on Fano orbifolds.} Ann. Sci. Ec. Norm. Sup. (4) {\bf 34} (2001), no 4, 525-556.

\bibitem[DP10]{DP10} J.P.~Demailly, N.Pali:{\it  Degenerate complex Monge-Amp\`ere equations over compact K\"ahler manifolds.} Internat. J. Math. {\bf 21} (2010) 357-405.
 

 \bibitem[EGZ08]{EGZ08}  P.~Eyssidieux, V.~ Guedj, A.~ Zeriahi: {\it A priori $L^{\infty}$-estimates for degenerate 
  complex Monge-Amp\`ere equations.} IMRN  (2008), Art. ID rnn 070, 8 pp.
 
 \bibitem[EGZ09]{EGZ09}  P.~Eyssidieux, V.~ Guedj, A.~ Zeriahi: { \it Singular K\"ahler-Einstein metrics }  J. Amer. Math. Soc. {\bf 22} (2009), 607-639. 
  
   \bibitem[EGZ11]{EGZ11}  P.~Eyssidieux, V.~ Guedj, A.~ Zeriahi: {\it  Viscosity solutions to Degenerate Complex Monge-Amp\`ere Equations.} 
  Comm. Pure Appl. Math.  {\bf 64}  (2011),  no. 8, 1059-1094. 
 
\bibitem[EGZ15]{EGZ15} P.Eyssidieux, V.Guedj, A.Zeriahi : {\it  Weak solutions to degenerate complex Monge-Amp\`ere flows I.} 
Math.Ann. {\bf 362} (2015),  931-963.

\bibitem[EGZ16]{EGZ16} P.Eyssidieux, V.Guedj, A.Zeriahi : { \it Weak solutions to degenerate complex Monge-Amp\`ere flows II.} 
Advances in Math. {\bf 293} (2016), 37-80.

\bibitem[GZ05]{GZ05} V.~ Guedj, A.~ Zeriahi: {\it Intrinsic capacities on compact K{\"a}hler  manifolds. } J. Geom. Anal.  {\bf 15}  (2005),  no. 4, 607-639.


  \bibitem[HL09]{HL09}  F.R.~Harvey, H.B.~Lawson: { \it Dirichlet duality and the nonlinear Dirichlet problem.}  Comm. Pure Appl. Math.  62  (2009),  no. 3, 396-443.


\bibitem[Kaw81]{Kaw1} Y. Kawamata {\it  Characterization of abelian varieties.} Compositio Mathematica
{\bf 43} (1981), 253--276.

\bibitem[Kaw00]{Kaw2} Y. Kawamata { \it Semipositivity, vanishing and applications.}  
 Lecture notes of a course given  in ICTP, Triest (2000). 
\bibitem[Kol97]{SP} J. ~Koll\'ar: {\em Singularities of pairs} in {\em Algebraic Geometry, Santa Cruz 1995} Proc. Symp. Pure Math. {\bf 62}, Amer. Math. Soc.  (1997). 
\bibitem[KM]{KM}  J. ~Koll\'ar , S. ~Mori : { \it Birational geometry of algebraic varieties.}  Cambridge Tracts in Math, {\bf 134} (1998), 254pp.


 
\bibitem[Pa08]{Paun08} M. ~P\u{a}un : {\it  Regularity properties of the degenerate complex Monge-Amp\`ere equations on compact K\"ahler manifolds}. 
Chin. Ann. Math. Ser. B {\bf 29} (2008), no6, 623-630.


 
\bibitem[PSS12]{PSS} D. H. ~Phong, J. Song, J. Sturm: {\it Complex Monge Amp\`ere Equations.} Surveys in Differential Geomety, vol. 17, 327-411 (2012). 


\bibitem[Rei87]{ypg} M. ~Reid : {\it  Young person's guide to canonical singularities} in {Algebraic Geometry Bowdoin 1985}, Proc. Symp. Pure Math. {\bf 46}, Amer. Math. Soc. (1987). 

\bibitem[ST07]{ST07} J.Song, G.Tian : {\it The K\"ahler-Ricci flow on surfaces of positive Kodaira dimension}. Invent. Math. {\bf 170} (2007) 609-653. 

\bibitem[ST12]{ST12} J.Song, G.Tian :{  \it Canonical measures and K\"ahler-Ricci flow.} 
  J. Amer. Math. Soc. {\bf 25} (2012), no. 2, 303-353. 

\bibitem[ST09]{ST09} J.Song, G.Tian : {\it  The K\"ahler-Ricci flow through singularities.} 
  Preprint (2009) arXiv:0909.4898
  
  
  \bibitem[SY12]{SY12} J. ~Song,  Y. ~Yuan : {\it The K\"ahler-Ricci flow on singular Calabi-Yau varieties.}
   Advances in geometric analysis, 119-137, Adv. Lect. Math., 21, Int. Press, 2012. 
   

\bibitem[TZ06]{TZ06} G. ~Tian, Z. ~Zhang: {\it  On the K\"ahler-Ricci flow on projective manifolds of general type}. Chinese Ann. Math. Ser. B {\bf 27} (2006) 179-192. 

 \bibitem[Tsu88]{Ts88} H. ~Tsuji : {\it Existence and degeneration of K\"ahler-Einstein metrics on minimal algebraic varieties of general type.} Math. Ann.  {\bf 281}  (1988),  no. 1, 123--133.
 
 

 \bibitem[Tos10]{Tos10} V. ~Tosatti: {\it Adiabatic limits of Ricci-flat K\"ahler metrics. J. Differential Geom.} {\bf 84} (2010), no.2, 42-53.
 
  \bibitem[TWY14]{TWY14} V. ~Tosatti, B. ~Weinkove, Y. ~Yang: {\it The K\"ahler-Ricci flow, Ricci-flat metrics and collapsing limits}. Preprint  arXiv:1408.0161
  
   \bibitem[TZ15]{TZ15} V. ~Tosatti, Y. ~Zhang: {\it Infinite-time singularities of the K\"ahler-Ricci flow.} Geom. Topol. {\bf 19} (2015), no. 5, 2925-2948.
  
  
 \bibitem[Yau78]{Yau78} S.T.~Yau: { \it On the Ricci curvature of a compact K\"ahler manifold and the complex Monge-Amp\`ere equation},
 Comm. Pure Appl. Math. {\bf 31}, 339-441 (1978). 
 

 \end{thebibliography}
\end{document}